\documentclass[12pt]{article}
\usepackage{amssymb,latexsym}
\usepackage{amsmath}
\usepackage{tikz}
\usepackage{verbatim}
\usetikzlibrary {positioning}
  \usetikzlibrary{arrows,topaths}
       \usetikzlibrary{decorations,backgrounds,snakes}
\usepackage{rawfonts}
\usepackage{amsmath, psfrag}
\usepackage{amsfonts}
\usepackage{amsthm}
\usepackage{lineno,hyperref}

\usepackage{verbatim}
\usepackage {tikz}
\usetikzlibrary {positioning}
\definecolor {processblue}{cmyk}{0.96,0,0,0}

\usepackage{latexsym}
\newtheorem{theorem}{Theorem}%[section]
\newtheorem{lemma}[theorem]{Lemma}

\newtheorem{proposition}[theorem]{Proposition}
\def\qed{\ifhmode\unskip\nobreak\fi\quad\ifmmode\Box\else$\Box$\fi}

\begin{document}
\title{The $6\times 6$ grid is $4$-path-pairable}
\author{{\sl Adam S. Jobson} \\ University of Louisville\\ Louisville, KY 40292
\and {\sl Andr\'e E. K\'ezdy}\\ University of Louisville\\ Louisville, KY 40292
\and {\sl Jen\H{o} Lehel} \\ University of Louisville\\ Louisville, KY 40292
\\
and 
\\
Alfr\'ed R\'enyi Mathematical Institute, \\
Budapest, Hungary}

\maketitle

\begin{abstract}
Let $G=P_6\Box P_6$ be the $6\times 6$ grid, the Cartesian product of two  
paths of six vertices. Let $T$ be the set of eight distinct vertices of $G$, called terminals, and assume that $T$ is partitioned into four terminal pairs 
$\{s_i,t_i\}$, $1\leq i\leq 4$. We prove that $G$ is $4$-path-pairable, that is, for every $T$ there 
exist in $G$ pairwise edge disjoint $s_i,t_i$-paths, $1\leq i\leq 4$.  \end{abstract}

\section{Introduction}

For $k$ fixed, a graph $G$ is {\it $k$-path-pairable}, if for any set of $k$ disjoint pairs of vertices, $s_i,t_i$, $1\leq i\leq k$, there exist pairwise edge-disjoint $s_i,t_i$-paths in $G$. The {\it path-pairability number}, denoted $pp(G)$,  is the largest $k$ such that $G$ is $k$-path-pairable.  

In \cite{pxp} we determine the path-pairability number of the grid graph $G(a,b)=P_{a}\Box P_{b}$, the Cartesian product of two paths  on $a$ and $b$ vertices, where there is an edge between
two vertices, $(i,j)$ and $(p,q)$, if and only if $|p-i|+|q-j|=1$, for $1\leq i\leq a$, $1\leq j\leq b$.  

\begin{theorem}[\cite{pxp}]  
\label{PxP} If  $k=\min\{a,b\}$, then
 $$pp(G(a,b))=\left\{
 \begin{array}{ccllll}
&k-1&  \hbox{ for } &k&=2,3,4\\
&3& \hbox{ for } &k&=5 \\
&4&  \hbox{ for } &k&\geq 6  \qquad 
\end{array}\right..$$
\end{theorem}

We complete the proof of the formula in Theorem \ref{PxP} by proving our main result:

\begin{theorem}
\label{main}$pp(G(6,6))=4$. 
\end{theorem}

In Section \ref{proof} the proof of Theorem \ref{main} is given in two parts. In Proposition \ref{66upper} we present a pairing of ten terminals which does not give a linkage
in $G(6,6)$. 
To  show that $G(6,6)$ is $4$-path-pairable Proposition \ref{66} uses a sequence of technical lemmas. These lemmas are listed next
in Section \ref{lemmas}, and they are proved separately in two notes, \cite{heavy} and \cite{escape}.
%%%%%%%%%%%%%%%%%%%%%%%%%%%%%%%

\section{Technical lemmas}
\label{lemmas}
Let $T=\{s_1,t_1,s_2,t_2,s_3,t_3,$ $s_4,t_4\}$ be the set of eight distinct vertices of the grid $G=P_6\Box P_6$, called {\it terminals}. The set $T$ is partitioned into
four terminal pairs,  $\pi_i=\{s_i,t_i\}$, $1\leq i\leq 4$, to be linked in $G$ by edge disjoint paths. A (weak) {\it linkage} for $\pi_i$,  $1\leq i\leq 4$, means a set of edge disjoint $s_i,t_i$-paths $P_i\subset G$.

The grid $G$ partitions
into four $P_3\Box P_3$ grids called {\it quadrants}.
We say that a set of  terminals in a quadrant $Q\subset G$ {\it escape} from $Q$
 if there are pairwise edge disjoint `mating paths' from the terminals into distinct mates (exits) located at the union of a horizontal  and a vertical boundary line of $Q$.
 A quadrant $Q\subset G$  is considered to be `crowded', if it contains $5$ or more terminals.
  Among the technical lemmas the proof of three lemmas pertaining to crowded quadrants was presented in \cite{heavy}. The technical lemmas for `sparse' quadrants containing at most  $4$ terminals are proved in \cite{escape}. 

  \subsection{Escaping from crowded quadrants}
  Let $A$ be a horizontal and  let $B$ be a vertical boundary line
of a quadrant $Q\subset G$, and for  a subgraph $S\subseteq G$ set $\|S\|=|T\cap S|$.

 \begin{lemma}
\label{heavy78}
 If $\|Q\|=7$ or $8$,
then there is
 a linkage for two or more pairs in $Q$, and there exist edge disjoint escape  paths for the unlinked terminals into distinct 
 exit vertices in $A\cup B$.\qed
\end{lemma}

 \begin{lemma}
   \label{heavy6}
If $\|Q\|=6$, then there is a linkage for one or more pairs in $Q$, and there exist edge disjoint escape paths for the unlinked terminals into distinct exit vertices of $A\cup B$ such that  $B\setminus A$ contains at most one exit.\qed
    \end{lemma}

  \begin{lemma}
   \label{heavy5}
 If $\|Q\|=5$ and $\{s_1,t_1\}\subset Q$,
then there is
 an $s_1,t_1$-path  $P_1\subset Q$, 
 and the complement of $P_1$ contains edge disjoint escape paths for  the three unlinked terminals into distinct exit vertices of $A\cup B$ such that $B\setminus A$ contains at most one exit.\qed
   \end{lemma}

    \subsection{Escaping from sparse quadrants}
 The vertices of a grid are represented as elements $(i,j)$ of a matrix arranged 
in rows $A(i)$ and columns $B(j)$. W.l.o.g. we may assume that $Q$ is the upper left quadrant of $G=P_6\Box P_6$, and thus
   $A=A(3)\cap Q$ and $B=B(3)\cap Q$ are the horizontal and vertical boundary lines of $Q$, respectively.
   
    For a vertex set $S\subset V(G)$ and a subgraph $H\subseteq G$, $H-S$ is interpreted as the subgraph obtained by the removal of $S$ and the incident edges from $H$;  $S$ is also interpreted as the subgraph of $G$ induced by $S$; $x\in H$ simply means a vertex of $H$.  Mating (or shifting) a terminal $w$  to vertex $w^\prime$, called a mate of $w$, means specifying a   $w,w^\prime$-path called a {\it mating path}.     
    
    Finding a linkage for two pairs are facilitated using the property of a graph being {\it weakly $2$-linked},
and by introducing the concept of a {\it frame}. 

A graph $H$ is weakly $2$-linked, if for every $u_1,v_1,u_2,v_2\in H$, not necessarily distinct vertices, there exist edge disjoint $u_i,v_i$-paths in $H$, for $i=1,2$. A weakly $2$-linked graph must be $2$-connected, but $2$-connectivity is not a sufficient condition.
 The next lemma lists a few weakly $2$-linked subgrids (the simple proofs are omitted). 
 \begin{lemma}
\label{w2linked}
The grid $P_3\Box P_k$,
and the subgrid of
$P_k\Box P_k$ induced by 
$(A(1)\cup A(2)) \cup$ $ (B(1)\cup B(2))$ is weakly $2$-linked, for $k\geq 3$.
 \qed
\end{lemma}

We use the $3$-path-pairability
of certain grids proved in \cite{pxp} (see in Theorem \ref{PxP}).
 \begin{lemma}
\label{3pp}
The grid $P_4\Box P_k$,
 is $3$-path-pairable, for $k\geq 4$.
 \qed
\end{lemma}

 		 Let $C\subset G$ be a cycle and let $x$ be a fixed vertex of $C$. Take two edge disjoint paths from a member
  of $\pi_j$ to $x$, for $j=1$ and $2$, not using edges of $C$. Then we say that
  the subgraph of the union of $C$ and the two paths to $x$ define a {\it frame} $[C,x]$, for $\pi_1,\pi_2$. A frame
  $[C,x]$, for $\pi_1,\pi_2$,  helps find a linkage
  for the pairs $\pi_1$ and $\pi_2$; in fact, it is enough to mate the other members of the terminal pairs onto $C$  using mating paths edge disjoint from $[C,x]$ and each other.

 The concept of a frame facilitates `communication' between quadrants of $G$. For this purpose frames in $G$ can be built on two standard cycles $C_0, C_1\subset G$ as follows.
  
 Let $C_0$ be the innermost $4$-cycle of $G$  induced by $(A(3)\cup A(4))\cap (B(3)\cup B(4))$, and let $C_1$ be the $12$-cycle around $C_0$ induced by the neighbors of $C_0$.
 Given a quadrant $Q$ we usually set $x_0=Q\cap C_0$ and we denote by $x_1$ the middle vertex of the path $Q\cap C_1$.  (For instance, in the upper right quadrant of $G$, $x_0=(3,4)$ and $x_1=(2,5)$.)

  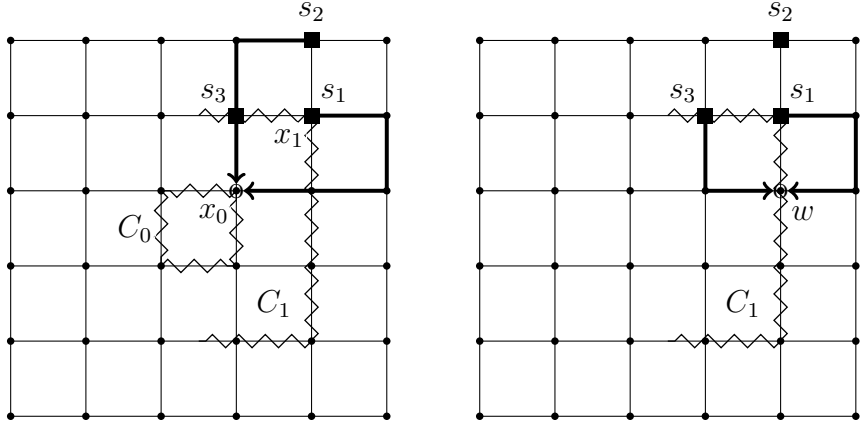
\begin{figure}[htp]
  \begin{center}
      \tikzstyle{T} = [rectangle, minimum width=.1pt, fill, inner sep=3pt]
 \tikzstyle{V} = [circle, minimum width=1pt, fill, inner sep=1pt]
\tikzstyle{B} = [rectangle, draw=black!, minimum width=4pt, inner sep=4pt]
\tikzstyle{txt}  = [circle, minimum width=1pt, draw=white, inner sep=0pt]
%\tikzstyle{Wedge} = [draw,double,line width=1pt,-,black!100]
\tikzstyle{Wedge} = [draw,line width=1.5pt,-,black!100]

\begin{tikzpicture}
 \foreach \x in {0,1,2,3,4,5} 
    \foreach \y in {0,1,2,3,4,5} 
      { \node[V] () at (\x,\y) {};}
      \foreach \u in {0,1,2,3,4} 
    \foreach \v in {1,2,3,4,5} 
   { \draw (\u,\v) -- (\u+1,\v); \draw (\u,\v-1) -- (\u,\v); } 
  \draw (0,0) -- (5,0) -- (5,5);	
     \draw[line width=.5pt,snake=zigzag]  (2,2) -- (3,2) -- (3,3) -- (2,3) -- (2,2);
      \draw[Wedge,->](4,4)--(5,4) -- (5,3) -- (3.1,3);
     \draw[Wedge,->] (4,5) -- (3,5) -- (3,3.1)  ;
\draw[line width=.5pt,snake=zigzag]  (2.5,4) -- (4,4) -- (4,1) -- (2.5,1);
     
     \node[T](s1) at (4,4){};
      \node[txt] () at   (4.3,4.3){$s_1$}; 
         \node[T,label=above:$s_2$]() at (4,5){};
           \node[T]() at (3,4){};
           \node[txt] () at   (2.7,4.3){$s_3$}; 
           \node[]() at (3,3){o}; 
       \node[txt] () at   (2.7,2.7){$x_0$}; 
          \node[txt] () at   (3.7,3.7){$x_1$}; 
        \node[txt] () at   (1.65,2.5){$C_0$};  
           \node[txt] () at   (3.5,1.5){$C_1$};        
\end{tikzpicture}
\hskip1truecm
\begin{tikzpicture}
 \foreach \x in {0,1,2,3,4,5} 
    \foreach \y in {0,1,2,3,4,5} 
      { \node[V] () at (\x,\y) {};}
      \foreach \u in {0,1,2,3,4} 
    \foreach \v in {1,2,3,4,5} 
   { \draw (\u,\v) -- (\u+1,\v); \draw (\u,\v-1) -- (\u,\v); } 
  \draw (0,0) -- (5,0) -- (5,5);	
 
      \draw[Wedge,->](4,4) -- (5,4) -- (5,3) -- (4.1,3);
     \draw[Wedge,->]  (3,4) -- (3,3) -- (3.9,3) ;
     \draw[line width=.5pt,snake=zigzag]  (2.5,4) -- (4,4) -- (4,1) -- (2.5,1);
     
        \node[T](s1) at (4,4){};
      \node[txt] () at   (4.3,4.3){$s_1$}; 
      \node[T,label=above:$s_2$]() at (4,5){}; 
         \node[T]() at (3,4){};
           \node[txt] () at   (2.7,4.3){$s_3$}; 
          \node[]() at (4,3){o};  
      \node[txt] () at   (4.3,2.7){$w$};  
       \node[txt] () at   (3.5,1.5){$C_1$};  
\end{tikzpicture}
\end{center}
\caption{Framing $[C_0,x_0]$ for $\pi_1,\pi_2$, and $[C_1,w]$, for $\pi_1,\pi_3$}
\label{frames}
		\end{figure}

 Let  $\alpha\in\{0,1\}$ be fixed, assume that there are two terminals in a quadrant $Q$ belonging to distinct pairs, say $s_1\in\pi_1$, $s_2\in \pi_2$,  and  let $w\in Q\cap C_\alpha$. We say that $[C_\alpha,w]$ is a {\it framing  in $Q$ for $\pi_1,\pi_2$ to $C_\alpha$}, if 
 there exist edge disjoint mating paths in $Q$
 from $s_1$ and from $s_2$ to $w$,  edge disjoint from $C_1$
 (see examples in Fig.\ref{frames} for framing in the upper right quadrant).
 
    \begin{lemma}
 \label{frame}  
 Let  $s_1\in\pi_1,s_2\in\pi_2$ be two (not necessarily distinct) terminals/mates  in a quadrant $Q$.
 
 (i) For any mapping $\gamma:\{s_1,s_2\}\longrightarrow\{C_0,C_1\}$,
 there exist edge disjoint  mating paths in $Q$ from $s_j$ to vertex $s_j^\prime\in\gamma(s_j)$, $j=1,2$, not using edges of $C_1$.
 
 (ii)  For any fixed $\alpha\in\{0,1\}$, there is a framing 
 $[C_\alpha,x_\alpha]$, 
 for $\pi_1, \pi_2$, where $x_\alpha\in C_\alpha\cap Q$ and the mating paths are in $Q$. 
  \end{lemma}

\begin{lemma}
\label{12toCa} Let $s_p,s_q,s_r$  be distinct terminals in a quadrant $Q$ belonging to three distinct pairs. Then there is a framing in $Q$ for $\pi_p,\pi_q$ to $C_\alpha$, for some $\alpha\in\{0,1\}$, and there is an edge disjoint mating path in $Q$ from $s_r$ to $C_\beta$, 
where $\beta=\alpha+1\pmod 2$, and edge disjoint from $C_1$.
\end{lemma}
\begin{lemma}
\label{Caforpq}
 Let $s_1,s_2,s_3$  be distinct  terminals in a quadrant $Q$ (belonging to distinct pairs); 
let $y_0\in Q$ be a corner vertex of $Q$ with degree three in $G$, and 
let  $z\in \{x_0,y_0\}$ be a fixed corner vertex of $Q$.  Then,

(i)  for some $1\leq p<q\leq 3$, there is a framing in $Q$ for $\pi_p,\pi_q$ to $C_0$,
 and there is an edge disjoint mating path in $Q$ from the third terminal to $C_1$;
 
(ii)  for some $1\leq p<q\leq 3$, there is a framing in $Q$ for $\pi_p,\pi_q$ to $C_1$,
 and there is an edge disjoint mating path in $Q$ from the third terminal to $z$;
\end{lemma}
\begin{lemma}
\label{exit}
Let $A$  be a boundary line of a quadrant $Q\subset G$. Let $Q_0$ be the subgraph obtained by removing the edges of $A$ from $Q$, and
let $Q_i$, $1\leq i\leq 4$, be one of the subgraphs in Fig.\ref{Qmin} obtained from $Q_0$ by edge removal and edge contraction.

(i) For any $H=Q_i$, $1\leq i\leq 4$, and for any three distinct terminals of $H$ there exist edge disjoint mating paths in $H$ from the terminals into not necessarily distinct vertices in $A$.
\begin{figure}[htp]
\tikzstyle{W} = [double,line width=.5,-,black!100]
%\tikzstyle{T} = [rectangle, minimum width=.1pt, fill, inner sep=2pt]
 \tikzstyle{V} = [circle, minimum width=1pt, fill, inner sep=1pt]
 \tikzstyle{T} = [circle, minimum width=1pt, fill, inner sep=1pt] 
  \tikzstyle{A}  = [rectangle, minimum width=1pt, draw=black,fill=white, inner sep=1.4pt]
\begin{center}
\begin{tikzpicture}
\draw[dashed] (0.5,0.5)--(2.3,0.5)--(2.3,0.9)--(.5,0.9)--(.5,0.5);
\foreach \x in {.7,1.4,2.1}\draw(\x,.7)--(\x,2.1);
\foreach \y in {1.4,2.1}\draw(.7,\y)--(2.1,\y);
\foreach \x in {.7,1.4,2.1}\foreach \y in {.7,1.4,2.1}\node[V] () at (\x,\y){};
\foreach \x in {.7,1.4,2.1}\node[A]()at(\x,.7){};
\node() at (.2,.7){A};
\node() at (1.4,0){$Q_0$};
\end{tikzpicture}
\hskip.5cm
\begin{tikzpicture}
%\draw[dashed] (0.5,1.2)--(2.3,1.2)--(2.3,1.6)--(.5,1.6)--(.5,1.2);
\draw (1.4,2.1)--(1.4,1.4)--(2.1,1.4) (.7,1.4)--(1.4,1.4) (.7,.7)--(.7,1.4) (1.4,.7)--(1.4,1.4) (2.1,.7)--(2.1,2.1);
\foreach \x in {.7,1.4,2.1}\foreach \y in {.7,1.4}\node[V] () at (\x,\y){};
\node[T]()at(2.1,2.1){};\node[T]()at(1.4,2.1){};
\foreach \x in {.7,1.4,2.1}\node[A]()at(\x,.7){};
%\node() at (2.7,1.4){N};
\node() at (1.4,0){$Q_1$};
\end{tikzpicture}
\hskip.5cm
\begin{tikzpicture}
%\draw[dashed] (0.5,1.2)--(2.3,1.2)--(2.3,1.6)--(.5,1.6)--(.5,1.2);
\draw (.7,1.4)--(.7,2.1)--(1.4,1.4) (2.1,1.4)--(.7,1.4) (.7,.7)--(.7,1.4) (1.4,.7)--(1.4,1.4) (2.1,.7)--(2.1,2.1);
\foreach \x in {.7,1.4,2.1}\foreach \y in {.7,1.4}\node[V] () at (\x,\y){};
\foreach \x in {.7,2.1}\node[T] () at (\x,2.1){};
\foreach \x in {.7,1.4,2.1}\node[A]()at(\x,.7){};
\node() at (1.4,0){$Q_2$};
\end{tikzpicture}
\hskip.5cm
\begin{tikzpicture}
%\draw[dashed] (0.35,2.15)--
%(1.2,1.2)--(2.3,1.2)--(2.3,1.6)-- (1.3,1.6)--
%(.8,2.3)--(0.35,2.15);
\draw (.7,2.1)--(2.1,2.1);
\foreach \x in {.7,1.4,2.1}\draw(\x,.7)--(\x,2.1);
\foreach \x in {1.4,2.1}\foreach \y in {.7,1.4}\node[V] () at (\x,\y){};
\foreach \y in {.7,2.1}\node[V] () at (.7,\y){};
\node[T]()at(1.4,2.1){};
\node[T]()at(2.1,2.1){};
\foreach \x in {.7,1.4,2.1}\node[A]()at(\x,.7){};
\node() at (1.4,0){$Q_3$};
\end{tikzpicture}
\hskip.5cm
\begin{tikzpicture}
%\draw[dashed] 
%(1.6,2.3)--(2.4,1.6)--(2.2,1)--(1.6,1.7)--(1.2,1.7)--(.6,1)
%--(.4,1.6)--(1.2,2.3)--(1.6,2.3);
\draw (.7,2.1)--(2.1,2.1);
\foreach \x in {.7,1.4,2.1}\draw(\x,.7)--(\x,2.1);
\foreach \x in {.7,2.1}\foreach \y in {.7,1.4}\node[V] () at (\x,\y){};
\foreach \y in {.7,2.1}\node[V] () at (1.4,\y){};
\node[T]()at(.7,2.1){};
\node[T]()at(2.1,2.1){};
\foreach \x in {.7,1.4,2.1}\node[A]()at(\x,.7){};
\node() at (1.4,0){$Q_4$};
%\node() at (0,1.4){N};
\end{tikzpicture}
\caption{Mating into $A$ in adjusted quadrants}
	\label{Qmin}
\end{center}
\end{figure}
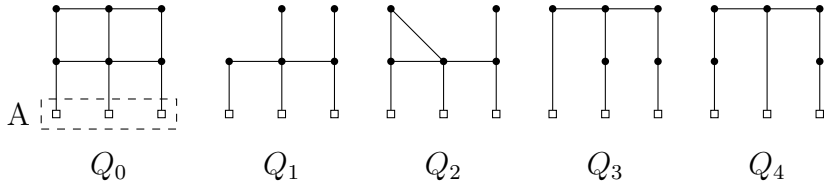

(ii) If  $s_1,t_1,s_2$ are not necessarily distinct terminals in $Q_0$ then there is
 an $s_1,t_1$-path  in $Q_0$ and an edge disjoint mating path from $s_2$ into a vertex of $A$.

(iii) From any three distinct terminals of $Q_0$ there exist pairwise edge disjoint mating paths 
into three distinct vertices of $A$. Furthermore, the claim remains true if two terminals not in $A$ coincide. 
\end{lemma}
\begin{lemma}
\label{heavy4}
Let $A,B$  be a horizontal and 
a vertical boundary line of quadrant $Q$, let $c$ be the corner vertex of $Q$ not in $A\cup B$, and let $b$ be the middle vertex of $B$ (see  $Q_0$ in Fig.\ref{except}). Denote by $Q_0$  the grid obtained by removing the edges of $A$ from $Q$, and let $T$ be a set of  at most four distinct terminals in $Q_0$. 

(i)  If  $T\subset Q_0-A$ and $c\notin T$, 
then for every terminal $s\in T$, there  is a linkage
in $Q_0$  to connect $s$ to $b$, and  there exist edge disjoint mating paths in $Q_0$
from the remaining terminals of $T$ into not necessarily distinct vertices of $A$. 

\begin{figure}[htp]
\tikzstyle{T} = [rectangle, minimum width=.1pt, fill, inner sep=2.5pt]
 \tikzstyle{V} = [circle, minimum width=1pt, fill, inner sep=1pt]
 \tikzstyle{A}  = [circle, minimum width=.5pt, draw=black, inner sep=2pt]
  \tikzstyle{B}  = [rectangle, minimum width=1pt, draw=black,fill=white, inner sep=1.4pt]
\begin{center}
\begin{tikzpicture}
\draw[dashed] (0.5,0.5)--(2.3,0.5)--(2.3,.9)--(.5,.9)--(.5,0.5);
\foreach \y in {1.4,2.1}\draw (.7,\y)--(2.1,\y);
\foreach \x in{.7,1.4,2.1}\draw (\x,.7)--(\x,2.1);
\foreach \x in {.7,1.4,2.1}\foreach \y in {1.4,2.1}\node[V] () at (\x,\y){};
\foreach \x in{.7,1.4,2.1}\node[B] () at (\x,.7){};
\node() at (2.1,2.4){B};
\node() at (.3,1){A};
\node() at (1.4,0){$Q_0$};
\node[B,label=left:$c$] () at(.7,2.1){};
\node[A] () at(2.1,1.4){};
\node() at (2.35,1.4){$b$};

\end{tikzpicture}
\hskip1cm
\begin{tikzpicture}
\foreach \y in {1.4,2.1}\draw (.7,\y)--(2.1,\y);
\foreach \x in{.7,1.4,2.1}\draw (\x,.7)--(\x,2.1);
\foreach \x in {.7,1.4,2.1}\foreach \y in {1.4,2.1}\node[V] () at (\x,\y){};
\foreach \x in{.7,1.4,2.1}\node[B] () at (\x,.7){};
\node[T,label=left:$s_{3}$] () at (.7,.7){};
\node[T,label=left:$s_{2}$] () at (.7,1.4){};
\node[T,label=left:$s_1$] () at (.7,2.1){};
 \node() at (1.4,0){$T_1$};
 \node[A] () at(2.1,1.4){};
 \node() at (2.35,1.4){$b$};
\end{tikzpicture}
\hskip1cm
\begin{tikzpicture}
\foreach \y in {1.4,2.1}\draw (.7,\y)--(2.1,\y);
\foreach \x in{.7,1.4,2.1}\draw (\x,.7)--(\x,2.1);
\foreach \x in {.7,1.4,2.1}\foreach \y in {1.4,2.1}\node[V] () at (\x,\y){};
\foreach \x in{.7,1.4,2.1}\node[B] () at (\x,.7){};
\node[T,label=above:$s_{2}$] () at (2.1,2.1){};
\node[T,label=above:$s_{3}$] () at (1.4,2.1){};
\node[T,label=above:$s_{4}$] () at (.7,2.1){};
\node[T] () at(.7,2.1){};
\node[T] () at(2.1,1.4){};
%\node() at (1.9,1.2){$b$}; 
 \node() at (1.4,0){$T_2$};
 \node() at (2.4,1.4){$s_1$};
\end{tikzpicture}
\caption{Projection to $A$}
	\label{except}
\end{center}
\end{figure}
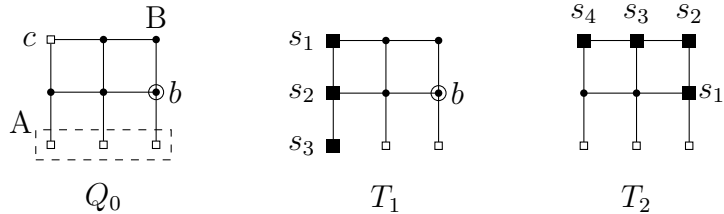
(ii) If $T$ is different from $T_1$ and 
$T_2$  in Fig.\ref{except}, then
for $\min\{3,|T|\}$ choices of a terminal $s\in T$, 
there is a linkage
in $Q_0$  to connect $s$ to $b$, and  there exist edge disjoint mating paths in $Q_0$
from the remaining terminals of $T$ into not necessarily distinct vertices of $A$.

(iii)
If $T$ is one of $T_1$ and 
$T_2$  in Fig.\ref{except}, then the claim  in (ii) above is true only for $s=s_1$ and $s_2$.
\end{lemma}
\begin{lemma}
\label{boundary}
Let $A,B$  be a horizontal and 
a vertical boundary line of a quadrant $Q$.
For every $s_1,t_1,s_2,s_3\in Q$  and $\psi: \{s_2,s_3\} \longrightarrow \{A,B\}$,
there is a linkage for $\pi_1$, and  there exist edge disjoint mating paths in $Q$
from $s_j$, $j=2,3$, to distinct vertices  $s_j^*\in \psi(s_j)$.
\end{lemma}
%%%%%%%%%%%%%%%%%%%%%%%%%%%%%%%%%
\section{Proof of Theorem \ref{main}}
\label{proof}

 \begin{proposition}
\label{66upper}
The $6\times 6$ grid is not $5$-path-pairable.
% $pp(G(6,6))\leq 4$.
 \end{proposition}
 \begin{proof}
Eight terminals are located in the upper left quadrant of $G=P_6\Box P_6$ as shown in Fig.\ref{cluster}; $t_1$ and $t_5$ are be placed anywhere in $G$. We claim that there is no linkage for $\pi_i$, $1\leq i\leq 5$. Assume on the contrary that there are pairwise edge disjoint $s_i,t_i$-paths, for $1\leq i\leq 5$. Then
 $P_1, P_2,P_3, $ and $P_5$ must leave the upper left $2\times 2$ square.
   \begin{figure}[htp]
  \begin{center}
  \tikzstyle{V} = [circle, minimum width=1pt, fill, inner sep=1pt]
\tikzstyle{T} = [rectangle, minimum width=.1pt, fill, inner sep=3pt]
\tikzstyle{B} = [rectangle, draw=black!, minimum width=1pt, fill=white, inner sep=1pt]
\tikzstyle{txt}  = [circle, minimum width=1pt, draw=white, inner sep=0pt]
\tikzstyle{Wedge} = [draw,line width=2.2pt,-,black!100]
\tikzstyle{M} = [circle, draw=black!, minimum width=1pt, inner sep=3.5pt]
		\begin{tikzpicture}	
\foreach \x in {1,...,3}\draw (\x,0.2)--(\x,3);
\foreach \y in {1,...,3}\draw (1,\y)--(3.8,\y);
\foreach \x in {1,...,3}\foreach \y in {1,...,3}\node[B]()at(\x,\y){};
\foreach \x in {1,2}\foreach \y in {1,2,3}\node[T]()at(\x,\y){};
\foreach \y in {2,3}\node[T]()at(3,\y){};

\draw[->,line width=1pt] (2,2)--(2,1.2);
\draw[->,line width=1pt] (1,1)--(1.8,1);
\draw[->,line width=1pt] (1,2)--(1,1.1);
\draw[->,line width=1pt] (1,1)--(1,.3);

\node[label=above:$s_1$]()at(.6,2.9){};
\node[label=left:$s_2$]()at(1,2){};
\node[label=left:$s_3$]()at(1,1){};
\node[label=right:$s_4$]()at(1.9,.8){};
\node[label=right:$t_4$]()at(2.9,1.8){};

\node[label=above:$t_3$]()at(2,3){};
\node[label=above:$t_2$]()at(3,3){};
\node[label=right:$s_5$]()at(1.9,1.8){};
\node[M]()at(2,1){};

\end{tikzpicture}		
\end{center}
 \caption{Unresolvable pairings}
 \label{cluster}
\end{figure}
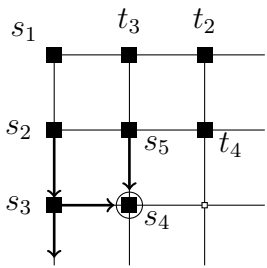

By symmetry, we may assume that $P_5$ starts with the edge $s_5 - (3,2)$, furthermore, either $P_1$ or $P_2$ must use the edge $(2,1) - (3,1)$. Then either $P_3$ or one of 
$P_1$ and $P_2$ uses the edge $(3,1) - (3,2)$. Thus a bottle-neck is formed at 
vertex $(3,2)$, since two paths are entering there and $P_4$ must leave it,  but only two edges, $(3,2) - (3,3)$ and $(3,2) - (4,2)$ are available, a contradiction.
 \end{proof}

\begin{proposition}
\label{66}
The $6\times 6$ grid is $4$-path-pairable.
\end{proposition}
\begin{proof}
We partition the grid $G=P_6\Box P_6$
 into four quadrants, 
named NW, NE, SW, SE according to their `orientation'. %Similarly, W, E, N, S will be the name of the left, right, upper, lower halfgrids, respectively. 
Given the terminal pairs $\pi_i=\{s_i,t_i\}\subset G$, $1\leq i\leq 4$,
a solution consists of pairwise edge disjoint $s_i,t_i$-paths, $P_i$, for $1\leq i\leq 4$, and is referred to as a {\it linkage} for $\pi_i$, $1\leq i\leq 4$.
Our procedure
described in terms of a tedious case analysis is based on the distribution of 
$T=\cup_{i=1}^4\pi_i$ in the four quadrants. The distributions of the terminals with respect to the quadrants are described with a so called {\it q-diagram} $\mathcal{D}$ defined as a (multi)graph with four nodes
labeled with the four
quadrants $Q_1, Q_2,Q_3, Q_4\subset G$, and containing four edges (loops and parallel edges are allowed): for each terminal pair 
$\{s_i,t_i\}\subset T$, $1\leq i\leq 4$, there is an edge
 $Q_aQ_b\in \mathcal{D}$  if and only if  $s_i\in Q_a$, $t_i\in Q_b$.
 
The proof is split into main cases A and B according to whether some quadrant contains a terminal pair, that is the diagram $\mathcal{D}$ is loopless, or  $\mathcal{D}$ contains a loop. \\

\noindent Case A: no quadrant of $G$ contains a terminal pair. Observe that in this case
the maximum degree in the q-diagram is at most $4$.

A.1: 
every quadrant has two terminals (the q-diagram is $2$-regular). There are four essentially different distributions, apart by symmetries of the grid, see in Fig.\ref{2222}. We may assume that $s_1,s_2\in NW$ and $s_3,s_4\in Q$, where $Q=$SE for the leftmost q-diagram and $Q=$NE for the other ones as indicated by the blackened nodes of the q-diagrams.% in Fig.\ref{2222}.

 \begin{figure}[htp]
  \begin{center}
      \tikzstyle{T} = [rectangle, minimum width=.1pt, fill, inner sep=3pt]
 \tikzstyle{V} = [circle, minimum width=1pt, fill, inner sep=1.5pt]
 \tikzstyle{Q} = [rectangle, draw=black!, minimum width=1pt, fill=white, inner sep=3pt]
\begin{tikzpicture}
\draw(.7,.7)--(.7,1.4)--(1.4,1.4)--(1.4,.7)--(.7,.7);
 \foreach \x in {0.7,1.4} 
    \foreach \y in {0.7,1.4}  {\node[Q] () at (\x,\y) {};}  
     \node[V] () at (.7,1.4) {};  \node[V] () at (1.4,.7) {};
          \node[label=left:{\small 1}]() at (.9,1.05){};
       \node[label=right:{\small 2}]() at (.7,1.6){};
\end{tikzpicture}
\hskip1cm
\begin{tikzpicture}
\draw(.7,1.4)--(1.4,.7)--(1.4,1.4)--(.7,.7)--(.7,1.4);
 \foreach \x in {0.7,1.4} 
    \foreach \y in {0.7,1.4}  { \node[Q] () at (\x,\y) {};}  
        \node[V] () at (.7,1.4) {};  \node[V] () at (1.4,1.4) {};
             \node[label=left:{\small 1}]() at (1,1.1){};
       \node[label=right:{\small 2}]() at (.6,1.25){};
\end{tikzpicture}
\hskip1cm
\begin{tikzpicture}
\draw(.65,.7)--(.65,1.4)(1.35,1.4)--(1.35,.7);
\draw(.75,.7)--(.75,1.4)(1.45,1.4)--(1.45,.7);
 \foreach \x in {0.7,1.4} 
    \foreach \y in {0.7,1.4}  { \node[Q] () at (\x,\y) {};}  
        \node[V] () at (.7,1.4) {};  \node[V] () at (1.4,1.4) {};
             \node[label=left:{\small 1}]() at (.9,1.05){};
       \node[label=right:{\small 2}]() at (.5,1.05){};
\end{tikzpicture}
\hskip1cm
\begin{tikzpicture}
\draw(.65,.7)--(1.35,1.4) (1.35,.7)--(.65,1.4);
\draw(.75,.7)--(1.45,1.4)(1.45,.7)--(.75,1.4);
 \foreach \x in {0.7,1.4} 
    \foreach \y in {0.7,1.4}  { \node[Q] () at (\x,\y) {};}  
        \node[V] () at (.7,1.4) {};  \node[V] () at (1.4,1.4) {};
                     \node[label=left:{\small 1}]() at (1.15,1.1){};
        \node[label=right:{\small 2}]() at (.6,1.3){};
\end{tikzpicture}
\end{center}
\caption{$\|Q\|=2$, for every quadrant}
\label{2222}
		\end{figure}
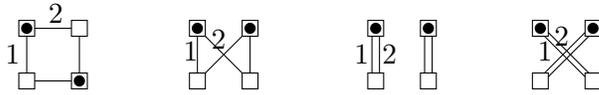
		
For each distributions we apply  Lemma \ref{frame} (ii) to obtain 
a framing in NW for $\pi_1,\pi_2$ to $C_0$ and another framing
in $Q$ for $\pi_3,\pi_4$ to $C_1$. Since the other two quadrants contain two-two terminals,  it is possible to
mate $t_1,t_2$ into vertices of $C_0$ and $t_3,t_4$ into vertices of $C_1$ by using Lemma \ref{frame} (i). Then the linkage is completed along the cycles $C_0$ and $C_1$.\\
 
 A.2:  the maximum degree of the q-diagram is $3$, and there is just one node with maximum degree, let $\|NW\|=3$.
 Fig.\ref{A12cases} lists q-diagrams with this property.
 \begin{figure}[htp]
\begin{center}
\tikzstyle{T} = [rectangle, minimum width=.1pt, fill, inner sep=3pt]
 \tikzstyle{V} = [circle, minimum width=1pt, fill, inner sep=1.5pt]
 \tikzstyle{Q} = [rectangle, draw=black!, minimum width=1pt, fill=white, inner sep=3pt]
\begin{tikzpicture}
\draw(.7,.7)--(.7,1.4)--(1.4,1.4)--(1.4,.7)--(.7,1.4);
 \foreach \x in {0.7,1.4} 
    \foreach \y in {0.7,1.4}  {\node[Q] () at (\x,\y) {};}  
      \node[V] () at (.7,1.4) {};  \node[V] () at (1.4,1.4) {};
      \node[label=left:{\small 1}]() at (.95,1.05){};
      \node[label=right:{\small 2}]() at (.6,1.05){};
      \node[label=above:{\small 3}]() at (1.05,1.2){};
     \node[label=right:{\small 4}]() at (1.15,1.05){};
\end{tikzpicture}
\hskip.6cm
\begin{tikzpicture}
\draw(.7,.7)--(.7,1.4)--(1.4,1.4)--(.7,.7) (.7,1.4)--(1.4,.7);
 \foreach \x in {0.7,1.4} 
    \foreach \y in {0.7,1.4}  {\node[Q] () at (\x,\y) {};}  
      \node[V] () at (.7,1.4) {};  \node[V] () at (1.4,1.4) {};
      \node[label=left:{\small 1}]() at (.95,1.1){};
      \node[label=right:{\small 2}]() at (.55,1.2){};
\end{tikzpicture}
\hskip.8cm
\begin{tikzpicture}
\draw  (.7,1.37)--(1.4,.6) (.7,1.53)--(1.4,.755)
 (.7,.7)--(1.4,1.4)--(.7,1.4);
 \foreach \x in {0.7,1.4} 
    \foreach \y in {0.7,1.4}  { \node[Q] () at (\x,\y) {};}   
       \node[V] () at (.7,1.4) {};  \node[V] () at (1.4,1.4) {}; 
   
         \node[label=left:{\small 1}]() at (1.15,1.1){};
       \node[label=right:{\small 2}]() at (.65,1.2){};  
\end{tikzpicture}
\hskip.6cm
\begin{tikzpicture}
\draw(.65,.7)--(.65,1.4) (.7,1.4)--(1.4,1.4)--(1.4,.7);
\draw (.75,.7)--(.75,1.4);
 \foreach \x in {0.7,1.4} 
    \foreach \y in {0.7,1.4}  { \node[Q] () at (\x,\y) {};}  
      \node[V] () at (.7,1.4) {};  \node[V] () at (1.4,1.4) {};    
           \node[label=left:{\small 1}]() at (.9,1.05){};
           \node[label=right:{\small 2}]() at (.5,1.05){};
\end{tikzpicture}
\hskip.6cm
\begin{tikzpicture}
\draw (.65,.7)--(.65,1.4) (.7,1.4)--(1.4,.7)--(1.4,1.4);
\draw (.75,.7)--(.75,1.4);
 \foreach \x in {0.7,1.4} 
    \foreach \y in {0.7,1.4}  { \node[Q] () at (\x,\y) {};}  
      \node[V] () at (.7,1.4) {};  \node[V] () at (1.4,.7) {};   
           \node[label=left:{\small 1}]() at (.9,1.){};
        \node[label=right:{\small 2}]() at (.5,1.){};
\end{tikzpicture}

\end{center}
\caption{$NW$ has $3$ terminals all other quadrants have less}
\label{A12cases}
\end{figure}
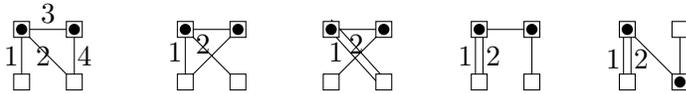
Let 
$s_1,s_2,s_3\in NW$, 
let $Q=$NE for the first four q-diagrams, and let $Q=$SE for the last q-diagram (see blackened nodes in Fig.\ref{A12cases}).
Applying
Lemma \ref{12toCa} with quadrant NW and $p=1, q=2$, we obtain a framing in NW for $\pi_1,\pi_2$ to $C_\alpha$, for some $\alpha\in \{0,1\}$, furthermore, we obtain a mating of $s_3$
into a vertex in $C_\beta\cap NW$, where $\beta=\alpha+1 \pmod 2$. Recall that the remaining quadrants contain at most two terminals. 
We use Lemma \ref{frame} (ii) with quadrant Q which yields a framing
in Q for $\pi_3,\pi_4$ to $C_\beta$.The solution is completed by mating the remaining terminals to the appropriate cycles applying Lemma \ref{frame} (i).\\

A.3: there are two quadrants containing three terminals, 
let 
$\|NW\|=\|Q\|=3$, where $Q=$NE or SE, see Fig.\ref{twomax3}.
 \begin{figure}[htp]
 \begin{center}
\tikzstyle{T} = [rectangle, minimum width=.1pt, fill, inner sep=3pt]
 \tikzstyle{V} = [circle, minimum width=1pt, fill, inner sep=1.5pt]
 \tikzstyle{Q} = [rectangle, draw=black!, minimum width=1pt, fill=white, inner sep=3pt]
 \begin{tikzpicture}%%%
\draw (.7,.7)--(1.4,.7);
\draw (.7,1.33)--(1.4,1.33)(.7,1.4)--(1.4,1.4) (.7,1.47)--(1.4,1.47);
\draw ;
 \foreach \x in {0.7,1.4} 
    \foreach \y in {0.7,1.4}  { \node[Q] () at (\x,\y) {};}    
         %\node[label=left:{\small 1}]() at (.9,1.05){};
       \node[label=below:{\small 4}]() at (1.05,1.3){};  
   \node[label=below:{\small (I)}]() at (1.05,.6){};        
\end{tikzpicture}
\hskip.6cm
 \begin{tikzpicture}
\draw  (.7,1.35)--(1.4,1.35) (.7,1.45)--(1.4,1.45)
(1.4,.7)--(1.4,1.4) (.7,.7)--(.7,1.4);
 \foreach \x in {0.7,1.4} 
    \foreach \y in {0.7,1.4}  { \node[Q] () at (\x,\y) {};}   
 \node[label=above:{\small 1}]() at (1.05,1.2){};
                 \node[label=above:{\small 2}]() at (1.05,.85){};
          \node[label=left:{\small 3}]() at (.9,1.05){};
       \node[label=right:{\small 4}]() at (1.15,1.05){};  
        \node[label=below:{\small (II)}]() at (1.05,.7){};      
\end{tikzpicture}
\hskip.6cm
\begin{tikzpicture}
\draw  (.7,1.37)--(1.4,.6) (.7,1.53)--(1.4,.755)
(1.4,.7)--(1.4,1.4) (.7,.7)--(.7,1.4);
 \foreach \x in {0.7,1.4} 
    \foreach \y in {0.7,1.4}  { \node[Q] () at (\x,\y) {};}    
         \node[label=left:{\small 3}]() at (.9,1.05){};
       \node[label=right:{\small 4}]() at (1.15,1.05){};  
        \node[label=below:{\small (III)}]() at (1.05,.6){};     
\end{tikzpicture}
\hskip.6cm
\begin{tikzpicture}
\draw  (.7,1.35)--(1.4,1.35) (.7,1.45)--(1.4,1.45)
(.7,1.4)--(1.4,.7)--(1.4,1.4);
 \foreach \x in {0.7,1.4} 
    \foreach \y in {0.7,1.4}  { \node[Q] () at (\x,\y) {};}   
        % \node[label=left:{\small 3}]() at (1.2,1.1){};
       \node[label=right:{\small 2}]() at (.7,1.25){};  
            \node[label=above:{\small 1}]() at (1.05,1.2){};  
             \node[label=below:{\small (IV)}]() at (1.05,.6){};  
                   \node[label=right:{\small 4}]() at (1.15,1.05){};    
\end{tikzpicture}
\hskip.6cm
\begin{tikzpicture}
\draw  (.7,1.35)--(1.4,1.35) (.7,1.45)--(1.4,1.45)
(.7,1.4)--(1.4,.7)(1.4,1.4)--(.7,.7);
 \foreach \x in {0.7,1.4} 
    \foreach \y in {0.7,1.4}  { \node[Q] () at (\x,\y) {};}   
         %\node[label=left:{\small 3}]() at (1.2,1.1){};
       \node[label=right:{\small 2}]() at (.7,1.25){};  
            \node[label=above:{\small 1}]() at (1.05,1.2){};  
             \node[label=below:{\small (V)}]() at (1.05,.6){};     
\end{tikzpicture}
\hskip.6cm
\begin{tikzpicture}
\draw  (.7,1.37)--(1.4,.6) (.7,1.53)--(1.4,.755)
 (1.4,.7)--(1.4,1.4)--(.7,1.4);
 \foreach \x in {0.7,1.4} 
    \foreach \y in {0.7,1.4}  { \node[Q] () at (\x,\y) {};}   
         %\node[label=left:{\small 1}]() at (1.15,1.1){};
      % \node[label=right:{\small 2}]() at (.65,1.2){};  
            \node[label=above:{\small 3}]() at (1.05,1.2){};
               \node[label=right:{\small 4}]() at (1.15,1.05){};  
                \node[label=below:{\small (VI)}]() at (1.05,.6){};     
\end{tikzpicture}
\hskip.6cm
\begin{tikzpicture}
\draw  (.65,1.37)--(1.35,.6) (.75,1.53)--(1.45,.755) 
(.72,1.42)--(1.42,.68) (.7,.7)--(1.4,1.4 );
 \foreach \x in {0.7,1.4} 
    \foreach \y in {0.7,1.4}  { \node[Q] () at (\x,\y) {};}   
         \node[label=left:{\small 4}]() at (1.35,.8){};
          \node[label=below:{\small (VII)}]() at (1.05,.6){};     
\end{tikzpicture}

\end{center}
\caption{$\|NW\|=\|Q\|=3$}
\label{twomax3}
\end{figure}
Let $s_1,s_2,s_3\in NW$, and $t_1,t_2\in Q$.

For the q-diagram (I) we define $G^*=G-(A(5)\cup A(6))$. We mate the terminals from row $A(4)$ to $s_4^\prime, t_4^\prime\in A(5)\cup A(6)$ along columns of $G$. 
Since $G^*\cong P_4\Box P_6$ is $3$-path-pairable
by Lemma \ref{3pp},
 there is a linkage for $\pi_1,\pi_2,\pi_3$ in $G^*$. Furthermore,
there is an edge disjoint $s_4^\prime,t_4^\prime$-path in the connected subgrid $A(5)\cup A(6)$ thus completing a solution. \\

For the q-diagrams (II) and (III) let $t_4\in Q$, where  $Q=$NE or SE, respectively.
We apply Lemma  \ref{exit} (iii) for  NW and for $Q$ with horizontal boundary line  $A=A(3)\cap NW$ and in 
$A=A(3)\cap NE$ or $A(4)\cap SE$,
respectively.
 \begin{figure}[htp]
 \begin{center}

\tikzstyle{A}  = [circle, minimum width=.5pt, draw=black, inner sep=2pt]
\tikzstyle{B} = [rectangle, draw=black!, minimum width=1pt, fill=white, inner sep=1pt]
\tikzstyle{T} = [rectangle, minimum width=.1pt, fill, inner sep=2.5pt]
 \tikzstyle{V} = [circle, minimum width=1pt, fill, inner sep=1pt]
\begin{tikzpicture}

\draw[line width=1.5pt] (1,2)--(3,2);
\draw[line width=1.5pt] (.5,2)--(.5,1.5)--(2,1.5)--(2,2);
\draw[->,double,snake](1.5,2)--(1.5,1.1);
\draw[->,double,snake](2.5,2)--(2.5,1.1);

\foreach \y in {.5,1,1.5,2,2.5,3} \draw (.5,\y)--(3,\y);
\foreach \x in {.5,1,1.5,2,2.5,3} \draw (\x,.5)--(\x,3);

  \foreach \x in {.5,1,1.5,2,2.5,3}\node[A]() at (\x,2){};  
   \foreach \x in {.5,1,1.5,2,2.5,3} \foreach \y in {.5,1,1.5,2,2.5,3}
     \node[V]() at (\x,\y){};
      \node() at (4.1,2){$A(3)$};
      \node() at (4.1,1.5){$A(4)$};
 \node() at (0.2,2.25){$s_2^\prime$};  
      \node() at (2.7,2.25){$t_4^\prime$}; 
        \node() at (1.3,2.3){$s_3^\prime$};  
                \node() at (.8,2.3){$s_1^\prime$};    
          \node() at (3.25,2.25){$t_1^\prime$};     
                   \node() at (2.22,2.25){$t_2^\prime$};        
\end{tikzpicture}
\begin{tikzpicture}
\draw[line width=1.5pt] (1,2)--(1,1.5)--(2,1.5);
\draw[line width=1.5pt] (.5,2)--(2.5,2)--(2.5,1.5);
\draw[->,double,snake](1.5,2)--(1.5,1.1);
\draw[->,double,snake](3,1.5)--(3,2.4);
\foreach \y in {.5,1,1.5,2,2.5,3} \draw (.5,\y)--(3,\y);
\foreach \x in {.5,1,1.5,2,2.5,3} \draw (\x,.5)--(\x,3);
   \foreach \x in {.5,1,1.5,2,2.5,3} \foreach \y in {.5,1,1.5,2,2.5,3}
     \node[V]() at (\x,\y){};
      \foreach \x in {.5,1,1.5}\node[A]() at (\x,2){};  
       \foreach \x in {2,2.5,3}\node[A]() at (\x,1.5){};  
     \node() at (2.7,1.3){$t_1^\prime$}; 
      \node() at (0.2,2.25){$s_1^\prime$};  
        \node() at (1.3,2.3){$s_3^\prime$};    
              \node() at (.8,2.3){$s_2^\prime$};   
          \node() at (3.25,1.3){$t_4^\prime$};  
      \node() at (2.2,1.3){$t_2^\prime$};          
\end{tikzpicture}
\end{center}
\caption{Cases (II) and (III)}
\label{A3A4}
\end{figure}
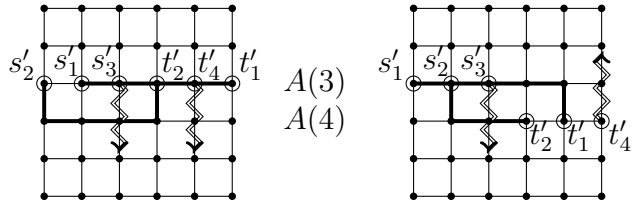
Thus we obtain six distinct mates 
 $s_1^\prime,s_2^\prime,s_3^\prime\in A(3)\cap NW$ and
  $t_1^\prime,t_2^\prime,t_4^\prime\in A(3)\cap NE$ or
   %$t_1^\prime,t_2^\prime,t_4^\prime\in 
   $A(4)\cap SE$, see the encircled vertices in Fig. \ref{A3A4}. 
Observe that the mating paths are edge disjoint from 
the $2\times 6$ grid $G^*=A(3)\cup A(4)$. The mating paths from
$s_3^\prime$ and $t_4^\prime$ can be extended into the neighboring quadrants containing $t_3$ and $s_4$ along the columns of $G$ (zigzag lines in Fig.\ref{A3A4}).
Furthermore, a linkage for $\pi_2$ can be completed by an $s_2^\prime, t_2^\prime$-path in $G^*$ not using edges of
 $A(3)$, and a linkage for $\pi_3$ can be completed by an $s_3^\prime, t_3^\prime$-path in $G^*$ not using edges of  $A(4)$. \\

The solution for the q-diagrams (IV) and (V) follows a similar strategy using Lemma \ref{heavy4}. Assume that as a result of applying Lemma \ref{heavy4} twice, for quadrants NW and NE, we find a common index $\ell\in\{1,2\}$, say $\ell=1$, that satisfies the following property: there exists a path in NW
from $s_1$ to $s_1^\prime=(2,3)$, and there exists a path in NE from $t_1$ to $t_1^\prime=(2,4)$, furthermore, terminals $s_2,s_3\in NW$, $t_2,t_4\in NE$ are mated into not necessarily distinct vertices $s_2^\prime,s_3^\prime,t_2^\prime,t_4^\prime\in A(3)$ using edge disjoint mating paths.
Now we complete a linkage for $\pi_1$ by adding the edge $s_1^\prime t_1^\prime\in A(2)$. Since the mating paths do not use edges of $A(3)$, a linkage for $\pi_2$ can be completed by adding an $s_2^\prime, t_2^\prime$-path in $A(3)$. Next we extend the mating paths from $s_3^\prime$ and $t_4^\prime$ into $s_3^{*},t_4^{*}\in A(4)$ along the columns of $G$. Since $G^*=G-(NW\cup NE)$
contains $t_3,s_3^*,s_4,t_4^{*}$ and since, by Lemma \ref{w2linked}, $G^*\cong P_3\Box P_6$ is weakly $2$-linked,  the linkage for $\pi_3,\pi_4$ can be completed in $G^*$. 

A common index $\ell\in\{1,2\}$ as above exists by the pigeon hole principle if one of the terminal set in NW or in NE is different from type $T_1$ in Fig.\ref{except}.  If the terminals in both quadrants are  of type $T_1$, then we have $s_1,s_2,s_3\in B(1)\cap NW$ and $t_1,t_2,t_4\in B(6)\cap NE$. Now we mate $s_1,s_2,t_1,t_2$ into the $3\times 4$
grid $G^\prime$ induced by $(NW\cup NE)\setminus (B(1)\cup B(6))$ along their rows, furthermore, we mate $s_3,t_4$ to vertices $s_3^*, t_4^*\in A(4)$  along their columns.  Since $G^\prime$ is $2$-path-pairable,  a linkage  for $\pi_1,\pi_2$ can be completed in $G^\prime$. Since the weakly $2$-linked
$G^*=G-(NW\cup NE)$ contains 
$s_3^*,t_3,s_4,t_4^*$, there are edge disjoint
ptahs from $s_3^*$ to $t_3$ and from $s_4$ to $t_4^*$ completeing
a linkage  in $G^*$ for $\pi_3$ and $\pi_4$.\\

 \begin{figure}[htp]
 \begin{center}

\tikzstyle{A}  = [circle, minimum width=.5pt, draw=black, inner sep=2pt]
\tikzstyle{B} = [rectangle, draw=black!, minimum width=1pt, fill=white, inner sep=1pt]
\tikzstyle{T} = [rectangle, minimum width=.1pt, fill, inner sep=2.5pt]
 \tikzstyle{V} = [circle, minimum width=1pt, fill, inner sep=1pt]
\begin{tikzpicture}

\draw[line width=1.5pt] (1.5,3)--(1.5,.5)--(3,.5);
\draw[line width=1.5pt] (1,3)--(1,1.5)--(3,1.5);
\draw[->,double](.5,3)--(1.9,3);
\draw[->,double](3,1)--(3,1.9);

\foreach \y in {.5,1,1.5,2,2.5,3} \draw (.5,\y)--(3,\y);
\foreach \x in {.5,1,1.5,2,2.5,3} \draw (\x,.5)--(\x,3);
 
\foreach \x in {.5,1,1.5,2,2.5,3} \foreach \y in {.5,1,1.5,2,2.5,3}
     \node[V]() at (\x,\y){};

  \foreach \x in {.5,1,1.5}\node[T]() at (\x,3){}; 
    \foreach \y in {.5,1,1.5}\node[T]() at (3,\y){}; 
      \node() at (3.3,1){$t_4$}; 
         \node() at (3.3,1.5){$t_1$}; 
            \node() at (3.3,.5){$t_2$}; 
          \node() at (3.3,2.15){$t_4^\prime$}; 
        \node() at (2.15,3.3){$s_3^\prime$};  
    \node() at (.5,3.3){$s_3$};    
      \node() at (1,3.3){$s_2$};    
        \node() at (1.5,3.3){$s_1$};    
\end{tikzpicture}
\hskip1cm
\begin{tikzpicture}

\draw[line width=1.5pt] (1,2)--(1,1)--(2,1) (1.5,3)--(1.5,1.5)--(3,1.5);
\draw[->, line width=1.2] (1.5,3)--(1.9,3);
\draw[->, line width=1.2] (2.5,1.5)--(2.5,1.9);
\foreach \y in {.5,1,1.5,2,2.5,3} \draw (.5,\y)--(3,\y);
\foreach \x in {.5,1,1.5,2,2.5,3} \draw (\x,.5)--(\x,3);
 
\foreach \x in {.5,1,1.5,2,2.5,3} \foreach \y in {.5,1,1.5,2,2.5,3}
     \node[V]() at (\x,\y){};

  \foreach \x in {2.5,3}\node[A]() at (\x,1.5){}; 
    \foreach \y in {3}\node[A]() at (1.5,\y){}; 
    \node[A]() at (2,1){};  \node[A]() at (1,2){}; 
     \node[A]() at (2,3){};  \node[A]() at (2.5,2){};   
         \node() at (2.3,.75){$t_1^\prime$}; 
            \node() at (2.75,1.27){$t_4^\prime$}; 
            \node() at (2.75,2.2){$t_4^*$};    
          \node() at (3.3,1.27){$t_2^\prime$}; 
        \node() at (1.3,3.3){$s_3^\prime$}; 
        \node() at (2.2,3.3){$s_3^*$};           
      \node() at (1.3,2.78){$s_2^\prime$};    
        \node() at (.75,2.3){$s_1^\prime$};    
\end{tikzpicture}

\end{center}
\caption{Case (VI)}
\label{VI}
\end{figure}
For the diagram (VI) suppose that $s_1,s_2,s_3\in A(1)\cap NW$ and  $t_1,t_2,t_4\in B(6)\cap SE$. Mate $s_3$ into
$s_3^\prime\in B(4)$ along $A(1)$, and mate $t_4$ into 
$t_4^\prime\in A(3)$ along $B(6)$. 
Since $s_3^\prime,t_3,s_4,t_4^\prime\in NE$, and NE is weakly $2$-linked, a linkage can be completed in NE for $\pi_3,\pi_4$. For the pairs $\pi_1,\pi_2$ a linkage can be obtained easily by taking shortest paths through SW as shown in the left of Fig.\ref{VI}.

Assume now that the terminals in one of the quadrants NW and SE is not of type $T_1$ as before. Then we apply Lemma \ref{heavy4} for NW with $A=B(3)\cap NW$ and $b=(3,2)$, and we apply Lemma \ref{heavy4}  for SE with $A=A(4)\cap SE$ and $b=(5,4)$. Then by the pigeon hole principle, we obtain a common index $\ell\in\{1,2\}$, say $\ell=1$, which satisfies: there exists a path in NW
from $s_1$ to $s_1^\prime=(3,2)$, and there exits a path in SE from $t_1$ to $t_1^\prime=(5,4)$, furthermore, terminals $s_2,s_3\in NW$, $t_2,t_4\in NE$ are mated into not necessarily distinct vertices $s_2^\prime,s_3^\prime\in B(3)$ and $t_2^\prime,t_4^\prime\in A(4)$ using edge disjoint mating paths. We take an $s_2^\prime,t_2^\prime$-path in $B(3)\cup A(4)$ to complete a linkage for $\pi_2$. Then
the mating paths to
$s_3^\prime,t_4^\prime$ are extended  into 
$s_3^*,t_4^*\in NE$. Since 
NE is weakly $2$-linked, a linkage can be completed there for $\pi_3,\pi_3$ (see on the right of Fig.\ref{VI}).\\

For q-diagram (VII)  we apply Lemma \ref{Caforpq} (i) with $Q=NW$ and $y_0=(1,3)$. W.l.o.g. we assume that there is a framing in $NW$
for $\pi_1,\pi_2$ with $C_1$ and a mating of $s_3$ into $y_0$. We extend this mating path to $s_3^*=(1,4)\in NE$.
Next we apply Lemma \ref{12toCa} with $Q=SE$, 
 and $p=1, q=2$. Thus we obtain a framing in $SE$ for $\pi_1,\pi_2$ to 
 $C_\alpha$,  for some $\alpha\in\{0,1\}$.
\begin{figure}[htp]
\begin{center}
 \tikzstyle{M}  = [circle, minimum width=.5pt, draw=black, inner sep=2pt]
\tikzstyle{B} = [rectangle, draw=black!, minimum width=1pt, fill=white, inner sep=1pt]
\tikzstyle{T} = [rectangle, minimum width=.1pt, fill, inner sep=2.5pt]
 \tikzstyle{V} = [circle, minimum width=1pt, fill, inner sep=1pt]

 \tikzstyle{C} = [circle, minimum width=1pt, fill=white, inner sep=2pt]
 \tikzstyle{A} = [circle, draw=black!, minimum width=1pt, fill, inner sep=1.7pt]
\tikzstyle{T} = [rectangle, minimum width=.1pt, fill, inner sep=3pt]
\tikzstyle{txt}  = [circle, minimum width=1pt, draw=white, inner sep=0pt]
\tikzstyle{Wedge} = [draw,line width=2.2pt,-,black!100]

\begin{tikzpicture}
      \draw[->,line width=1.2pt] (0,5) -- (0,4) -- (.9,4);
        \draw[->,line width=1.2pt] (1,5) -- (1,4.1);
           \draw[->,line width=1.2pt] (4,2) -- (3.1,2);
    \draw[->,line width=1.2pt]  (0,3) -- (0,3) -- (2,3) -- (2,4.9); 
      \draw[->,line width=1.2pt]  (2,5) -- (2.9,5); 
        \draw[->,line width=1.2pt]  (5,2) -- (5,0) -- (3,0) -- (3,1.9); 
         \draw[->,line width=1.2pt] (5,1)--(4.1,1);
 \draw[snake,line width=.5pt]  (1,2) -- (3,2) -- (3,4)  -- (1,4) -- (1,2)        (3.1,5) -- (4,5) -- (4,1);
         
  \draw[dashed] (0,4) -- (0,0) -- (3,0) (4,0)--(4,1)--(0,1) (5,2)--(5,5)--(4,5);	
   \draw[dashed] (0,5) -- (2,5)  (3,4)--(5,4) (2,3)--(5,3) (0,2)--(1,2) (4,2)--(5,2);
   	  \draw[dashed] (1,0) -- (1,2)  (2,0)--(2,3);
 \foreach \x in {0,1,2,3,4,5} 
    \foreach \y in {0,1,2,3,4,5} 
      { \node[B] () at (\x,\y) {};}
      
    \node[T](s4) at (3,4){};
    \node[txt]() at (3.35,4.3) {$s_4$};
     \node[T](t4) at (2,1){};   
     \node[txt]() at (1.7,.7) {$t_4$};
   \node[T,label=right:$t_3$](t3) at (5,1){};   
  \node[T](t1) at (4,2){};  
  \node[txt]() at (4.35,1.7) {$t_1$};   
    \node[txt]() at (4.3,.7) {$t_3^\prime$};       
         \node[T,label=right:$t_2$](t2) at (5,2){};
               \node[T,label=above:$s_1$](s1) at (0,5){};
          \node[T,label=above:$s_2$](s2) at (1,5){};
               \node[T,label=left:$s_3$](s3) at (0,3){};        
               \node(s3*) at (3,5.35){$s_3^*$};
             \node[M] () at (1,4) {};       
            \node[M] () at (3,2) {};  \node[M] () at (4,1) {}; 
              \node[M,label=above:$y_0$] () at (2,5) {}; 
              \node[M] () at (3,5) {}; 
   
       \node[txt]() at (1.4,3.6) {$D$};
          
              \node[txt]() at (3.6,3.6) {$P_3^*$};
\end{tikzpicture}
\begin{tikzpicture}
\draw[double,line width=.3pt]  (1,2) -- (3,2) -- (3,4)  -- (1,4) -- (1,2)     (3,5) -- (4,5) -- (4,1) -- (5,1);
\draw[double,line width=.3pt]  (0,3) -- (2,3) -- (2,5)  -- (3,5)
     (0,5) -- (0,4) -- (1,4)  (1,5)--(1,4);
  \draw[double,line width=.3pt]  (4,2) -- (3,2) (5,2) -- (5,0)--(3,0)--(3,2);
  \draw[line width=2.2pt]  (4,5) -- (5,5) -- (5,3) -- (2,3) -- (2,0);  
       \draw[line width=2.2pt]  (3,5) -- (3,4) -- (5,4);  
  \draw[line width=2.2pt]  (1,1)-- (1,2) -- (0,2) -- (0,0) -- (2,0) -- (2,2);  
   \draw[dashed] (0,5)--(2,5) (0,4)--(0,2) (5,3)--(5,2)--(4,2) (0,1)--(4,1);      
  \draw[dashed] (0,1)--(1,1) (4,0)--(4,1) (1,0)--(1,1) (2,0)--(3,0);
 \foreach \x in {0,1,2,3,4,5} 
    \foreach \y in {0,1,2,3,4,5} 
      { \node[B] () at (\x,\y) {};}

        \node[B]()at (4,3) {};
         \node[T](s4) at (3,4){};
         \node[txt]() at (3.35,4.3) {$s_4$};
         \node[T](t4) at (2,1){};
           \node[txt]() at (1.7,.7) {$t_4$};
   \node[T,label=right:$t_3$](t3) at (5,1){};   
         \node[T](t1) at (4,2){};  
          \node[txt]() at (4.35,1.7) {$t_1$};
          
         \node[T,label=right:$t_2$](t2) at (5,2){};
     \node[T,label=above:$s_1$](s1) at (0,5){};  
              \node[T,label=above:$s_2$](s2) at (1,5){};
         \node[T,label=left:$s_3$](s3) at (0,3){};   
             \node[txt]() at (2.6,5.3) {$P_3$};
               \node[txt]() at (2.6,3.6) {$P_2$};
                       \node[txt]() at (1.4,2.4) {$P_1$};
\end{tikzpicture}

\end{center}
\caption{Case (VII)}
\label{VII}
\end{figure}
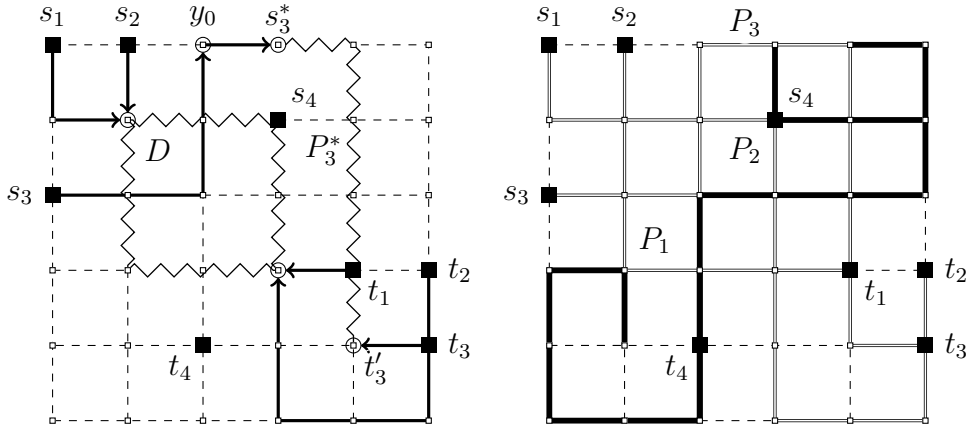

 For $\alpha =1$,  a linkage for $\pi_1,\pi_2$ is completed along $C_1$ and $t_3$ is mated in SE to $(4,4)\in C_0$. It remains to build a framing in $NE$ for $\pi_3,\pi_4$ with $C_0$. For this purpose we apply Lemma \ref{frame} with $s_3^*,s_4\in NE$ and mate $t_4$ in SW to $C_0$ not using edges of $C_1$.

For  $\alpha =0$, the solution is obtained by combining the frames as follows.  Let $D$ be  the $8$-cycle spanned by the neighbors of $(3,3)$ (see the left of Fig.\ref{VII}). Observe that no edges of $D$ have been used by the mating paths in the two framing. Thus a linkage $P_1,P_2$ for $\pi_1,\pi_2$ is completed
around $D$. A linkage $P_3$ for $\pi_3$ can be completed by a path 
$P^*_3\subset 
A(1)\cup B(5)$ from $s_3^*$ to $t_3^\prime=(5,5)$. The right picture in Fig.\ref{VII} shows that $s_4$ and $t_4$ are not disconnected by the linkage built so far, the tree highlighted in the picture 
saturates all vertices of $NE\cup SW$ and edge disjoint  from $P_1\cup P_2\cup P_3$. 
Hence 
 there is a  linkage for $\pi_4$.\\

A.4:  $\|NW\|=4$, let $s_1,s_2,s_3,s_4\in NW$.
Two cases will be distinguished according to whether there is a quadrant $Q\neq$NW with three or more terminals or not. 
 By symmetry, we may assume that $\|NE\|\geq \|SW\|$. 

A.4.1: $\|Q\|\geq 3$, where  $Q=$NE or SE.

 In each case we apply Lemma \ref{heavy4} twice: for $NW$ with  $A=A(3)\cap NW$, $B=B(3)\cap NW$, then for $Q$,  with $A=A(3)\cap Q$,
$B=B(3)\cap Q$, if $Q=$NE or
 with 
 $A=B(4)\cap Q$, $B=A(4)\cap Q$,
 if $Q=$SE. 
 
 Assume that there is a common index $\ell$, $1\leq j\leq 4$, resulting from the two applications of Lemma  \ref{heavy4}, such that 
$s_\ell$ is linked to $(2,3)$ in $NW$ 
and $t_\ell\in NE$ is linked in $NE$ to $(2,3)$ or
$t_\ell\in SE$ is  linked in SE to $(4,4)$, furthermore, the remaining (five or six) terminals are mated into $A(3)\cap NW$ and
into $A(3)\cap NW$ or $B(4)\cap SE$.

First we complete a linkage for $\pi_\ell$ by the inclusion of the edge  $(2,3)-(2,4)$.  W.l.o.g. assume that $\ell=1$. Lemma \ref{heavy4} also implies that the mating paths leading from $s_2,s_3,s_4$ to the not necessarily distinct mates 
$s_2^\prime,s_3^\prime,s_4^\prime$  are not using the edges of 
$A(3)$, and similarly, the mating paths in $Q$ to the not necessarily distinct mates
$t_i^\prime\in Q$ are not using the edges $A(3)$ or $B(4)$, for $Q=$NE or SE. 
 
  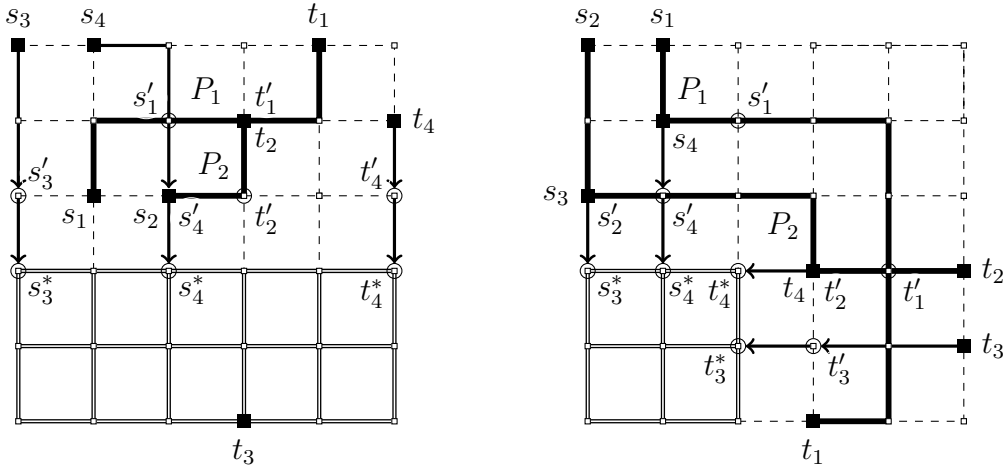
\begin{figure}[htp]
\begin{center}
  \tikzstyle{M}  = [circle, minimum width=.5pt, draw=black, inner sep=2pt]
\tikzstyle{B} = [rectangle, draw=black!, minimum width=1pt, fill=white, inner sep=1pt]
\tikzstyle{T} = [rectangle, minimum width=.1pt, fill, inner sep=2.5pt]
 \tikzstyle{V} = [circle, minimum width=1pt, fill, inner sep=1pt]

 \tikzstyle{C} = [circle, minimum width=1pt, fill=white, inner sep=2pt]
 \tikzstyle{A} = [circle, draw=black!, minimum width=1pt, fill, inner sep=1.7pt]
\tikzstyle{txt}  = [circle, minimum width=1pt, draw=white, inner sep=0pt]
\tikzstyle{Wedge} = [draw,line width=2.2pt,-,black!100]

\begin{tikzpicture}
  \draw[double,line width=.5pt]  (0,0) -- (0,2)--(5,2)--(5,0)--(0,0); 
 \draw[double,line width=.5pt] (3,0) -- (3,2); 
 \draw[double,line width=.5pt](4,2)--(4,0);
 \draw[double,line width=.5pt]    (0,1) -- (5,1) (0,2) -- (5,2);
 \draw[double,line width=.5pt]  (1,0) -- (1,2) (2,0) -- (2,2);
      
         \draw[->,line width=1.2pt] (0,5)--(0,3.1);   
         \draw[->,line width=1.2pt] (0,3)--(0,2.1); 
         \draw[->,line width=1.2pt] (1,5)--(2,5)--(2,3.1);   
         \draw[->,line width=1.2pt] (2,3)--(2,2.1); 
        % \draw[->,line width=1.2pt] (3,4)--(3,3.1);  
               \draw[->,line width=1.2pt] (5,4)--(5,3.1); 
                \draw[->,line width=1.2pt] (5,3)--(5,2.1); 
                   
              \draw[line width=2.2pt] (1,3)--(1,4)--(4,4)--(4,5);
             \draw[line width=2.2pt] (2,3)--(3,3)--(3,4);
             
   \draw[dashed] (5,3)--(3,3)-- (3,2) (3,5)-- (3,4) (0,3)--(2,3);
    \draw[dashed] (0,4)--(1,4)--(1,5)--(0,5) (1,2)--(1,3) (2,5)--(5,5)--(5,4)--(4,4)--(4,2) ;
          
 \foreach \x in {0,1,2,3,4,5} 
    \foreach \y in {0,1,2,3,4,5} 
      { \node[B] () at (\x,\y) {};}

        \node[B]()at (4,3) {};
                               
\node[M]()at (2,4) {};\node[M]()at (3,3) {};
\node[M]()at (0,3) {};\node[M]()at (5,3) {};
                                  
       \node[T](s2) at (2,3){};  
        \node[txt]() at (.3,3.3) {$s_3^\prime$};     
      \node[T,label=above:$s_3$](s3) at (0,5){};   
  \node[T,label=above:$s_4$](s4) at (1,5){};
  \node[T](s1) at (1,3){};   \node[txt](s1) at (.75,2.7) {$s_1$};  
   \node[txt]() at (2.3,2.7) {$s_4^\prime$};   
   \node[txt]() at (3.3,2.7) {$t_2^\prime$};  
     \node[txt]() at (4.7,3.3) {$t_4^\prime$};  
     \node[txt](s2) at (1.7,2.7) {$s_2$};  
     
      \node[txt](s1') at (1.7,4.3) {$s_1^\prime$};   
   \node[txt](t1') at (3.3,4.3) {$t_1^\prime$};  
   
     \node[T](t2) at (3,4){};    
     \node() at (3.3,3.75){$t_2$};
   \node[T,label=above:$t_1$](t1) at (4,5){};  
   \node[T,label=right:$t_4$](t4) at (5,4){};   
    \node[T,label=below:$t_3$](t3) at (3,0){}; 
     
    \node(s1*) at (2.3,1.75){$s_4^*$};  \node[M] () at (2,2) {}; 
     \node(s2*) at (.3,1.75){$s_3^*$}; \node[M]()at (0,2) {};
   \node(t1*) at (4.7,1.7){$t_4^*$}; \node[M]()at (5,2) {};
       
    \node[txt]() at (2.6,3.4) {$P_2$};
     \node[txt]() at (2.5,4.4) {$P_1$};                       
                                         
\end{tikzpicture}
\hskip1cm
\begin{tikzpicture}
 \draw[double,line width=.5pt]    (0,0) -- (2,0) (0,1)--(2,1)
 (0,2)--(2,2);
 \draw[double,line width=.5pt] (0,0) -- (0,2) 
 (1,0) -- (1,2) (2,0) -- (2,2);
      
         \draw[->,line width=1.2pt] (0,3)--(0,2.1); 
         \draw[->,line width=1.2pt] (1,3)--(1,2.1);   
         \draw[->,line width=1.2pt] (5,1)--(3.1,1);  
          \draw[->,line width=1.2pt] (3,1)--(2.1,1);     
           \draw[->,line width=1.2pt] (1,4)--(1,3.1);   
            \draw[->,line width=1.2pt] (3,2)--(2.1,2);       
   \draw[line width=2.2pt] (1,5)--(1,4)--(4,4)--(4,0)--(3,0);
             \draw[line width=2.2pt] (0,5)--(0,3)--(3,3)--(3,2)--(5,2);
             
 \draw[dashed] (0,5)--(5,5)--(5,0)--(4,0) (3,5)-- (5,5)-- (5,4) (1,4)--(1,5) (2,5)--(2,2);
 \draw[dashed] (0,4)--(0,3)--(1,3) (4,4)--(5,4) (4,5)--(4,2) 
 (3,5)--(3,3)--(5,3) (2,0)--(3,0)--(3,2) (0,4)--(1,4);
          
 \foreach \x in {0,1,2,3,4,5} 
    \foreach \y in {0,1,2,3,4,5} 
      { \node[B] () at (\x,\y) {};}

        \node[B]()at (4,3) {};
     \node[T](s4) at (1,4){};   \node[txt]() at (1.3,3.7) {$s_4$}; 
       \node() at (1.3,2.7) {$s_4^\prime$}; 
          \node() at (2.3,4.3) {$s_1^\prime$}; 
   \node[T](t4) at (3,2){};  \node()at(2.75,1.75) {$t_4$}; 
     \node[T,label=left:$s_3$](s3) at (0,3){};   
     \node(s3') at (.3,2.7) {$s_2^\prime$}; 
      
       \node[T,label=below:$t_1$](t1) at (3,0){};            
       \node[T,label=right:$t_2$](t2) at (5,2){}; 
       %\node[txt](t2) at (3.35,1.3) {$t_2$}; 
        \node[txt](t3') at (3.35,.7) {$t_3^\prime$}; 
          
       \node[txt](t1') at (4.35,1.7) {$t_1^\prime$};   
          \node[txt](t2') at (3.3,1.7) {$t_2^\prime$};  
      \node[T,label=above:$s_2$](s2) at (0,5){};   
  \node[T,label=above:$s_1$](s1) at (1,5){};
         \node[T,label=right:$t_3$](t3) at (5,1){};   
  
\node[M]()at (0,2) {};  \node[M]()at (1,2) {};  
\node[M]()at (4,2) {};  \node[M]()at (2,2) {};  
  \node[M]()at (2,4) {};   \node[M]()at (3,1) {}; 
     \node[M]()at (2,1) {};
    \node(s4*) at (1.25,1.75){$s_4^*$};  
     \node(s3*) at (.3,1.75){$s_3^*$}; 
     \node[M]()at (1,3) {};
     \node(t3*) at (1.7,.7){$t_3^*$}; 
       \node(t3*) at (1.758,1.7){$t_4^*$}; 
  
 \node[txt]() at (2.6,2.6) {$P_2$};
\node[txt]() at (1.4,4.4) {$P_1$};                                                             
\end{tikzpicture}

\end{center}
\caption{$\|NW\|=4$, $\|Q\|\geq 3$}
\label{NW4}
\end{figure}

Next a  linkage for another pair 
$\pi_j$, $2\leq j\leq 4$ is completed along $A(3)$ (if $Q=$NE) or along $A(3)\cup B(4)$ (if $Q=$SE), where $j$ is selected as follows:  $j$ is arbitrary provided all mates are distinct; $j$ is an index if $s_j^\prime\in A(3)$ or $t_j^\prime\in A(3)$ or $t_j^\prime\in B(4)$
is the only vertex hosting two mates in $NW$ (or in $Q$); if both NW and $Q$ contain repeated mates,
then $j\in\{2,3,4\}$ is selected to satisfy that both $s_j^\prime$ and $t_j^\prime$ are repeated mates (such index $j$ exists by the pigeon hole principle). W.l.o.g. let $j=2$.

Finally, a linkage  can be obtained  by extending (three or four) mating paths
from the remaining distinct mates into neighbors in $SW\cup SE$ (if $Q=$NE) or into neighbors in SW (if $Q=$SE). Then the linkage for
$\pi_3,\pi_4$ can be completed in the $3\times 6$
 grid $SW\cup SE$ or in the quadrant SW which are both weakly-$2$-linked, by Lemma \ref{w2linked} (Fig.\ref{NW4} shows solutions).
 
 Therefore a solution is obtained once a common index $\ell$ can be selected
 to link $\pi_\ell$ as above. By the pigeon hole principle there is a common index $\ell$ unless the terminal set in NW is of type $T_2$, and
 the terminal set in $Q$ is of type $T_1$ or $T_2$ in Fig.\ref{except} (ii). 
 
 We handle the exceptional cases one-by-one. 
 Let $s_1=(2,3), s_2=(1,3)$, $s_3=(1,2)$ and $s_4=(1,1)$. 
 
 For $\|SE\|=4$, if the terminals in NW  are located according to type $T_2$ as well, then the argument using the common index $\ell$ can be repeated by 
 switching the role of 
 NE and SE. Since the pattern $T_2$ is not symmetric about the diagonal, the solution above works.
 
  \begin{figure}[htp]
\begin{center}
  \tikzstyle{M}  = [circle, minimum width=.5pt, draw=black, inner sep=2pt]
\tikzstyle{B} = [rectangle, draw=black!, minimum width=1pt, fill=white, inner sep=1pt]
\tikzstyle{T} = [rectangle, minimum width=.1pt, fill, inner sep=2.5pt]
 \tikzstyle{V} = [circle, minimum width=1pt, fill, inner sep=1pt]

 \tikzstyle{C} = [circle, minimum width=1pt, fill=white, inner sep=2pt]
 \tikzstyle{A} = [circle, draw=black!, minimum width=1pt, fill, inner sep=1.7pt]
\tikzstyle{txt}  = [circle, minimum width=1pt, draw=white, inner sep=0pt]
\tikzstyle{Wedge} = [draw,line width=2.2pt,-,black!100]

\begin{tikzpicture}
 \foreach \x in {0,...,5} 
    \foreach \y in {0,1,2,3,4,5} 
      { \draw  (0,\y)--(5,\y); \draw (\x,0)--(\x,5);}  
      
 \draw[line width=2.2pt] (2,5)--(4,5)
 (1,5)--(1,4)--(3,4)--(3,5) ;
   \draw[line width=2.2pt] (0,5)--(0,2)--(3,2)--(3,4)
 (2,4)--(2,3)--(5,3)--(5,5) ;
             
 \foreach \x in {0,1,2,3,4,5} 
    \foreach \y in {0,1,2,3,4,5} 
      { \node[B] () at (\x,\y) {};}                        
\node[M]() at (5,3){}; \node[M]() at (3,2){}; 
\node[M]() at (0,2){}; \node[M]() at (2,3){}; 
\node() at (5.3,2.7) {$t_1^\prime$};  
   \node() at (1.7,2.7) {$s_1^\prime$};                        
 \node[T](s1) at (2,4){};   \node() at (1.75,3.75) {$s_1$}; 
        
  \node[T,label=above:$s_4$](s4) at (0,5){};   
  \node[T,label=above:$s_3$](s3) at (1,5){};    
     \node[T,label=above:$s_2$](s2) at (2,5){};  
 
     \node[T,label=above:$t_2$](t2) at (4,5){};    
  
   \node[T,label=above:$t_1$](t1) at (5,5){};  
   \node[T,label=above:$t_3$](t3) at (3,5){};   
    \node[T](t4) at (3,4){}; \node() at (3.25,3.75) {$t_4$};
  \node() at (3.25,1.75) {$t_4^\prime$};
   \node() at (.3,1.75) {$s_4^\prime$};
      \node() at (.4,2.4) {$P_4$};   
     \node() at (4.6,3.4) {$P_1$};  
     \node() at (3.6,4.6) {$P_2$};   
     \node() at (1.4,4.4) {$P_3$}; 
                                 
\end{tikzpicture}

\end{center}
\caption{$\|NE\|=4$}
\label{diagonal2}
\end{figure}

 For $\|NE\|=4$, we have $\{t_3,t_4\}=\{(1,4),(2,4)\}$ and $\{t_1,t_2\}=\{(1,5),(1,6)\}$. Thus there are
 two pairs are in $A(1)$, say 
 $\pi_2,\pi_3\subset A(1)$. Their linkage can be done 
 using  the $s_2,t_2$-path $P_2\subset A(1)$
 and the $s_3,t_3$-path $P_3\subset (B(2)\cup A(2)\cup B(4))$. The remaining terminals can be mated along their distinct columns  into vertices $s_1^\prime,t_1^\prime\in A(3)$ and $s_4^\prime,t_4^\prime\in A(4)$. The linkage
 for $\pi_1,\pi_4$ can be completed  along $A(3)$ and $A(4)$, respectively (see Fig.\ref{diagonal2}).

  \begin{figure}[htp]
\begin{center}
  \tikzstyle{M}  = [circle, minimum width=.5pt, draw=black, inner sep=2pt]
\tikzstyle{B} = [rectangle, draw=black!, minimum width=1pt, fill=white, inner sep=1pt]
\tikzstyle{T} = [rectangle, minimum width=.1pt, fill, inner sep=2.5pt]
 \tikzstyle{V} = [circle, minimum width=1pt, fill, inner sep=1pt]

 \tikzstyle{C} = [circle, minimum width=1pt, fill=white, inner sep=2pt]
 \tikzstyle{A} = [circle, draw=black!, minimum width=1pt, fill, inner sep=1.7pt]
\tikzstyle{txt}  = [circle, minimum width=1pt, draw=white, inner sep=0pt]
\tikzstyle{Wedge} = [draw,line width=2.2pt,-,black!100]

\begin{tikzpicture}
 \foreach \x in {0,...,5} 
    \foreach \y in {0,1,2,3,4,5} 
      { \draw  (0,\y)--(5,\y); \draw (\x,0)--(\x,5);}  
      
         \draw[->,line width=1.2pt] (5,3)--(5,2.1);   
      
 \draw[double,line width=.5pt] (1.7,5.6)--(5.7,5.6)-- (5.7,2.7)--(1.7,2.7)--(1.7,5.6);
 \draw[double,line width=.5pt] (1.3,5.6)--(-.4,5.6)--(-.4,-.7)--(5.7,-.7) --(5.7,2.35)--(1.3,2.35)--(1.3,5.6) ;
        
 \foreach \x in {0,1,2,3,4,5} 
    \foreach \y in {0,1,2,3,4,5} 
      { \node[B] () at (\x,\y) {};}                        
\node[M](t3') at (5,2){}; 
\node() at (5.3,2) {$t_3^\prime$};                            
 \node[T](s1) at (2,4){};   \node() at (2.3,3.7) {$s_1$}; 
        
  \node[T,label=above:$s_4$](s4) at (0,5){};   
  \node[T,label=above:$s_3$](s3) at (1,5){};    
     \node[T,label=above:$s_2$](s2) at (2,5){};  
 
     \node[T,label=right:$t_2$](t2) at (5,4){};    
  
   \node[T,label=right:$t_1$](t1) at (5,5){};  
   \node[T,label=below:$t_4$](t4) at (2,0){};   
    \node[T,label=right:$t_3$](t3) at (5,3){}; 
  
      \node[txt]() at (5.4,.4) {$G^\prime$};                            
\end{tikzpicture}
\hskip1cm
\begin{tikzpicture}
 \foreach \x in {0,...,5} 
    \foreach \y in {0,1,2,3,4,5} 
      { \draw  (0,\y)--(5,\y); \draw (\x,0)--(\x,5);}  
          
 \draw[double,line width=.5pt] (1.7,5.6)--(5.7,5.6)-- (5.7,-.7)--(3.7,-.7)--(3.7,3.55)--(1.7,3.55)--(1.7,5.6);
\draw[double,line width=.5pt] (1.3,5.6)--(-.4,5.6)--
(-.4,-.7)--(3.3,-.7) --(3.3,3.35)--(1.3,3.35)--(1.3,5.6) ;
        
 \foreach \x in {0,1,2,3,4,5} 
    \foreach \y in {0,1,2,3,4,5} 
      { \node[B] () at (\x,\y) {};}                        

%\node() at (5.3,2) {$t_3^\prime$};                            
 \node[T](s1) at (2,4){};   \node() at (2.3,3.7) {$s_1$}; 
        
  \node[T,label=above:$s_4$](s4) at (0,5){};   
  \node[T,label=above:$s_3$](s3) at (1,5){};    
     \node[T,label=above:$s_2$](s2) at (2,5){};  
 
     \node[T,label=below:$t_2$](t2) at (4,0){};    
  
   \node[T,label=below:$t_1$](t1) at (5,0){};  
   \node[T,label=below:$t_3$](t3) at (3,0){};   
    \node[T](t4) at (2,2){};   \node() at (2.3,2.3){$t_4$}; 
  
      \node[txt]() at (5.4,.4) {$G^\prime$};                            
\end{tikzpicture}

\end{center}
\caption{}
\label{Qexcept}
\end{figure}

For $\|NE\|=3$ we have  $\{t_1,t_2\}=\{(1,6),(2,6)\}$ and $(3,6)=t_3$ or $t_4$. First we mate the terminal at $(3,6)$ to $(4,6)$. The grid
$G^\prime=(SW\cup SE)\cup (B(1)\cup B(2))$ is 
 $2$-path pairable, thus  a linkage for $\pi_3,\pi_4$ can be completed in $G^\prime$. In $G-G^\prime$ which is $2$-path-pairable as well, there is a linkage for $\pi_1,\pi_2$ (see the left of Fig. \ref{Qexcept}).

For $\|SE\|=3$ we have  $\{t_1,t_2\}=\{(6,5),(6,6)\}$ and $(6,4)=t_3$ or $t_4$. The $2$-path pairable grid
$G^\prime=(A(1)\cup A(2))\cup (B(5)\cup B(6))\setminus (B(1)\cup B(2))$ and its complement are both $2$-path pairable. Thus $G^\prime$ contains  a linkage for $\pi_1,\pi_2$, and $G-G^\prime$ contains a linkage for $\pi_3,\pi_4$ (see the right of Fig. \ref{Qexcept}).\\
 
In the remaining cases we have  $\|Q\|\leq 2$, for every $Q\neq$ NW.
 Since $2\geq \|NE\|\geq \|SW\|$, we have either 
 $\|SE\|=2$ and $\|NE\|=\|SW\|= 1$ or $\|NE\|=2$.\\
 
  A.4.2: $\|SE\|=2$ and $\|NE\|=\|SW\|= 1$.
 We apply Lemma \ref{heavy4} for NW with 
 $A=A(3)\cap NW$ and $B=B(3)\cap NW$. There are at least two terminals that can be mapped into $(2,3)$, hence by the pigeon hole principle,  there is an index $1\leq \ell\leq 4$ such that $s_\ell$ is mated to $s_\ell^\prime=(2,3)$ and 
$t_\ell\in  NE\cup SE$. W.l.o.g. we may assume that $\ell=1$, and the  terminals
$s_2,s_3,s_4$ are mated into $s_2^\prime,s_3^\prime,s_4^\prime\in A(3)$ by the lemma.  
If $t_1\in NE$, then a linkage for $\pi_1$ is completed by an $s_1^\prime,t_1$-path. Moreover, if 
 $s_2^\prime,s_3^\prime,s_4^\prime$ are distinct
 then  the linkage for the remaining terminals can be completed 
  in the $3$-path-pairable grid $G^*=G-(A(1)\cup A(2))$. 
  
  Assume now that $s_2^\prime,s_3^\prime,s_4^\prime$ are not distinct, 
 let  $w\in A(3)\cap NW$ be the mate of two terminals of NW, that is $s_i^\prime=s_j^\prime=w$, for some $2\leq i<j\leq 4$ (actually, one of them is a terminal, $s_i^\prime=s_i$ or $s_j^\prime=s_j$). Since $\|SW\|\leq 1$,
 $t_i$ or $t_j$ is a terminal in $NE\cup SE$, say $t_i\in NE\cup SE$; let $t_k$ be the third terminal in $NE\cup SE$ (that is $t_1,t_i,t_k\in N$). 
 
 We plan to specify a linkage for $\pi_1$ by mating $t_1$ to $s_1^\prime$, then specify a linkage for $\pi_i$
 by mating $t_i$ to a vertex of $A(3)$; the remaining terminals
 can be mated into the weakly $2$-linked SW and find there a linkage for $\pi_j,\pi_k$. The plan is easy to realize provided $t_1\in NE$.  It is enough to mate $t_i\in SE$ along its column to $t_i^\prime\in A(3)$, then $s_j^\prime,s_k^\prime\in A(3)$ to $s_j^*,s_k^*\in A(4)$, and to mate $t_k\in SE$ along its row to $t_k^*\in B(3)$.
 
 Assume now that 
 $t_1\in SE$.
We introduce three auxiliary terminals in NE,
 let $x=(2,4)$, $x^\prime=(3,5)$, and $y=(3,4)$.  There exist an $x,x^\prime$-path $X$ and 
 an edge disjoint  path $Y$ from the  terminal of NE to $y$ not using edges of $A(3)$. If the terminal of NE is $t_k$, and thus $t_1,t_i\in SE$, then we extend $Y$ to 
 $t_k^*=(4,3)$ by adding the path $y-(4,4)-t_k^*$, furthermore, we mate
 $t_1$ to $t_1^\prime=(4,5)$ and we mate $t_i$ to $t_i^\prime=(4,6)$ (see the left of Fig.\ref{adhoc}).
 If the terminal of NE is $t_i$, and thus $t_1,t_k\in SE$, 
 let $t_i^\prime=y$ be the mate of $t_i$, we mate
 $t_1$ to $t_1^\prime=(4,5)$, and we mate $t_k$ to 
 $t_k^\prime=(6,3)$. 
 
 In each case we complete a linkage for $\pi_1$ by adding the path $X$ and the two edges $s_1^\prime x$ and $t_1^\prime x^\prime$; and we complete a linkage for $\pi_i$ by adding the $s_i^\prime, t_i^\prime$-path in $A(3)$. The unpaired terminals/mates from
 $A(3)\cup B(4)$ are mated
  into the weakly $2$-linked quadrant SW to complete a solution. An example is shown in the left of Fig.\ref{adhoc}. 
 
 \begin{figure}[htp]
\begin{center}
  \tikzstyle{M}  = [circle, minimum width=.5pt, draw=black, inner sep=2pt]
\tikzstyle{B} = [rectangle, draw=black!, minimum width=1pt, fill=white, inner sep=1pt]
\tikzstyle{T} = [rectangle, minimum width=.1pt, fill, inner sep=2.5pt]
 \tikzstyle{V} = [circle, minimum width=1pt, fill, inner sep=1pt]

 \tikzstyle{C} = [circle, minimum width=1pt, fill=white, inner sep=2pt]
 \tikzstyle{A} = [circle, draw=black!, minimum width=1pt, fill, inner sep=1.7pt]
 \begin{tikzpicture}   
 \draw[line width=2.2pt] (0,4)--(3,4)(4,3)--(4,2)--(4,1)--(3,1);
  \draw[line width=2.2pt] (1,3)--(5,3)--(5,2);
  \draw[snake](3,4)--(4,4)--(4,3) (4,5)--(3,5)--(3,3);
 
\foreach \x in {0,1,2} \draw[double,line width=.5pt] (\x,2)--(\x,0);
 \foreach \y in {0,1,2} \draw[double,line width=.5pt] (0,\y)--(2,\y);  
 \draw[->,line width=1.2pt] (3,0)-- (5,0)--(5,1.9);  
    \draw[->,line width=1.2pt] (0,5)--(0,3.1);
     \draw[->,line width=1.2pt] (0,3)--(0,2.1);

    \draw[->,line width=1.2pt] (1,5)--(1,3.1);
     \draw[->,line width=1.2pt] (1,3)--(1,2.1);
    \draw[->,line width=1.2pt] (3,3)--(3,2)--(2.1,2);
    
\draw[dashed]  (0,5)--(3,5) (4,4)--(4,5)--(5,5)--(5,3)(4,0)--(4,1)--(5,1); 
\draw[dashed]  (2,5)--(2,2) (2,0)--(3,0)--(3,2)--(5,2)(2,1)--(3,1)(4,4)--(5,4); 

 \foreach \x in {0,1,2,3,4,5} 
    \foreach \y in {0,1,2,3,4,5} 
      { \node[B] () at (\x,\y) {};}
                                
          \node[T](t1) at (3,1){};  \node()at(3.3,.7){$t_1$};                                
     \node[T,label=above:$t_k$](tk) at (4,5){}; 
 \node[T](ti') at (5,2){}; \node() at (5.3,1.75) {$t_i^\prime$}; 
       \node[T,label=below:$t_j$](tj) at (0,0){};    
      \node[T,label=below:$t_i$](ti) at (3,0){}; 
        \node[T,label=above:$s_k$](sk) at (0,5){}; 
        \node[T,label=left:$s_1$](s1) at (0,4){};  
             \node[T,label=above:$s_j$](sj) at (1,5){}; 
                             
          \node[T](si) at (1,3){};  \node() at (.7,2.7) {$s_i$};  
    \node[M]()at (4,3) {};   \node() at (4.3,3.3) {$x^\prime$}; 
         \node[M]()at (3,4) {};   \node() at (3.3,4.3) {$x$};  
    \node[M]()at (3,3) {};   \node() at (3.25,3.25) {$y$};   
            \node() at (.7,3.3) {$w$}; 
            \node[M,label=left:$s_k^\prime$]()at (0,3) {}; 
            
\node[M]()at (2,4) {};  \node() at (1.7,4.3) {$s_1^\prime$}; 
\node[M]()at (4,2) {};   \node() at (4.3,1.7) {$t_1^\prime$};  \node[T](s3) at (1,3){}; \node()at(1.35,3.3){$s_j^\prime$}; 
\node[M]()at (1,2) {};   \node() at (1.35,2.3) {$s_j^*$}; 
\node[M]()at (2,2) {};   \node() at (2.35,2.3) {$t_k^*$}; 

\node[M,label=left:$s_k^*$]()at (0,2) {};                         
       \node() at (3.65,3.65) {$X$};    
           \node() at (2.7,4.7) {$Y$};                   
            \node() at (.4,.4) {$G^*$};                
\end{tikzpicture}                     
\hskip1cm
\begin{tikzpicture}

\foreach \x in {0,...,5} \draw[double,line width=.5pt] (\x,3)--(\x,0);
 \foreach \y in {0,...,3} \draw[double,line width=.5pt] (0,\y)--(5,\y);
  \draw[line width=2.2pt] (0,5)--(5,5)--(5,4);
  
  \draw[->,line width=1.2pt] (2,4)--(0,4)--(0,3.1);
  \draw[->,line width=1.2pt] (1,5)--(1,3.1);
   \draw[->,line width=1.2pt] (2,5)-- (2,3.1); 
   \draw[->,line width=1.2pt] (4,5)-- (4,3.1); 

 \draw[dashed]  (0,5)--(0,4) (2,4)--(5,4)--(5,3) (3,5) -- (3,3); 
  
 \foreach \x in {0,1,2,3,4,5} 
    \foreach \y in {0,1,2,3,4,5} 
      { \node[B] () at (\x,\y) {};}
  
       \node[T,label=right:$t_4$](t2) at (5,4){};   
       \node[T](t4) at (2,2){};    \node()at(2.3,1.7){$t_2$};
                               
     \node[T,label=below:$t_1$](t1) at (5,0){};                                           
     \node[T,label=above:$t_3$](t3) at (4,5){};
    
        \node[T,label=above:$s_1$](s1) at (2,5){}; 
        \node[T,label=above:$s_4$](s4) at (0,5){};  
         \node[T,label=above:$s_3$](s3) at (1,5){};                 
         \node[T](s2) at (2,4){};  \node() at (1.7,4.3) {$s_2$}; 

\node[M]()at (2,3) {};      \node() at (1.7,3.3) {$s_1^{*}$}; 
\node[M]()at (1,3) {};      \node() at (.7,3.3) {$s_3^*$}; 
\node[M]()at (4,3) {};      \node() at (3.7,3.3) {$t_3^*$}; 
\node[M]()at (0,3) {};      \node() at (-.3,3.3) {$s_2^*$}; 
                         \node() at (4.6,4.6) {$P_4$};                                          
                         \node() at (.4,.4) {$G^*$};                     
\end{tikzpicture}

\end{center}
\caption{$\|NW\|=4$}
\label{adhoc}
\end{figure}

A.4.3:  $\|NE\|=2$. The solution starts with Lemma \ref{heavy4} applied for NW with $A=A(3)\cap NW$ and $B=B(3)\cap NW$. Since $\|NW\|=4$, there are three terminals that can be mated to $b=(2,3)$ and the other ones to $A$ unless  the four terminals in NW are located according to type $T_2$ in Fig.\ref{except}. First we sketch a solution for this exceptional case when
 $\{s_1,s_2\}=\{(1,3),(2,3)\}$ and $\{s_3,s_4\}=\{(1,1),(1,2)\}$, furthermore, the two terminals in NE are $t_3,t_4$. The solution on the right of Fig.\ref{adhoc} 
starts with a linkage $P_4$ for $\pi_4$ not using edges of NW not in $A(1)\cap NW$, and mating the other terminals of  $NW\cup NE$ into distinct vertices 
$s_1^*,s_2^*,s_3^*,t_3^*\in A(3)$. The linkage is completed for $\pi_1,\pi_2,\pi_3$ in the $3$-path-pairable $G^*=G-(A(1)\cup A(2))$.

Therefore, we may assume that the terminals in NW are not in position $T_2$, and when we apply Lemma \ref{heavy4} for NW with 
 $A=A(3)\cap NW$ and $B=B(3)\cap NW$, there are three terminals one can map into $(2,3)$. By the pigeon hole principle,  there is an index $1\leq \ell\leq 4$ such that $s_\ell$ is mated to $s_\ell^\prime=(2,3)$ and 
$t_\ell\in  NE\cup SE$. W.l.o.g. we may assume that $\ell=1$, and the  terminals
$s_2,s_3,s_4$ are mated into $s_2^\prime,s_3^\prime,s_4^\prime\in A(3)$ by the lemma.  If 
  $s_2^\prime,s_3^\prime,s_4^\prime$ are distinct, then we follow the solution given for the particular case above. 
  A linkage for $\pi_1$ is specified first, then 
   the $3$-path-pairability of $G^*=G-(A(1)\cup A(2))$ is used to obtain a linkage for the remaining pairs.
  
 \begin{figure}[htp]
\begin{center}
  \tikzstyle{M}  = [circle, minimum width=.5pt, draw=black, inner sep=2pt]
\tikzstyle{B} = [rectangle, draw=black!, minimum width=1pt, fill=white, inner sep=1pt]
\tikzstyle{T} = [rectangle, minimum width=.1pt, fill, inner sep=2.5pt]
 \tikzstyle{V} = [circle, minimum width=1pt, fill, inner sep=1pt]

 \tikzstyle{C} = [circle, minimum width=1pt, fill=white, inner sep=2pt]
 \tikzstyle{A} = [circle, draw=black!, minimum width=1pt, fill, inner sep=1.7pt]
\tikzstyle{txt}  = [circle, minimum width=1pt, draw=white, inner sep=0pt]
\tikzstyle{Wedge} = [draw,line width=2.2pt,-,black!100]
\hskip1cm
 \begin{tikzpicture}
\foreach \x in {0,1,2} \draw[double,line width=.5pt] (\x,2)--(\x,0);
 \foreach \y in {0,1,2} \draw[double,line width=.5pt] (0,\y)--(2,\y);
   
 \draw[line width=2.2pt] (4,5)--(4,4)--(2,4);
  \draw[line width=2.2pt] (5,2)--(5,3)--(1,3);
  
    \draw[->,line width=1.2pt] (0,5)--(0,3.1);
  \draw[->,line width=1.2pt] (1,5)--(1,3.1);
  
\draw[dashed] (0,5)--(5,5)--(5,4)(4,1)--(4,0)--(5,0)--(5,2);
 \draw[dashed] (0,4)--(2,4)--(2,2)--(5,2) (2,0)--(4,0);    
  \draw[dashed] (3,5)--(3,0) (4,1)--(5,1) 
  (0,3)--(1,3)(5,3)--(5,4);
  
  \draw[->,line width=1.2pt]  (1,3)--(1,2.1);
     \draw[->,line width=1.2pt] (4,3)--(4,2.1);
    \draw[->,line width=1.2pt] (4,2)--(4,1)-- (2.1,1);
   \draw[->,line width=1.2pt] (0,3)--(0,2.1);
    \draw[->,line width=1.2pt] (5,4)--(4,4)--(4,3.1);
    
 \foreach \x in {0,1,2,3,4,5} 
    \foreach \y in {0,1,2,3,4,5} 
      { \node[B] () at (\x,\y) {};}
                                
     \node[M]() at (4,3){};  \node()at(4.3,2.7){$t_k^\prime$};                                
     \node[T,label=above:$t_1$](t1) at (4,5){}; 
 \node[T](tk) at (5,4){}; \node() at (5.3,3.75) {$t_k$}; 
       \node[T,label=below:$t_j$](tj) at (0,0){};   
  \node[T](ti) at (5,2){}; \node() at (5.3,1.75) {$t_i$};         
    
        \node[T,label=above:$s_k$](sk) at (0,5){}; 
        \node[T,label=left:$s_1$](s1) at (0,4){};  
         \node[T,label=above:$s_j$](sj) at (1,5){};                 
          \node[T](si) at (1,3){};  \node() at (.7,2.7) {$s_i$};   
            \node() at (.7,3.3) {$w$}; 
\node[M]()at (2,4) {};  \node() at (1.7,4.3) {$s_1^\prime$}; 
\node[M]()at (3,4) {};  \node() at (2.7,4.3) {$t_1^\prime$}; 
 \node[M]()at (3,1) {}; \node[M]()at (4,2) {};  
\node[M,label=left:$s_k^\prime$]()at (0,3) {};     
\node[M]()at (5,3) {};  \node() at (5.3,2.7) {$t_i^{\prime}$};  \node[M]()at (2,1) {};   \node()at(2.3,.7){$t_k^*$};
\node[T](sj) at (1,3){}; \node() at (1.35,3.3) {$s_j^\prime$};     
\node[M]()at (1,2) {};   \node() at (1.35,2.3) {$s_j^*$}; 
\node[M,,label=left:$s_k^*$]()at (0,2) {};          
                     
       \node[txt]() at (3.6,4.4) {$P_1$};  
        %\node[txt]() at (4.6,3.4) {$P_i$};                      
                        \node[txt]() at (2.5,2.5) {$P_2$};
                         \node[txt]() at (.4,.4) {$G^*$};                
\end{tikzpicture}   
\hskip.3cm
\begin{tikzpicture}

  \draw[line width=2.2pt] (0,4)--(4,4)--(4,5);
   \draw[snake]  (0,3)--(0,1)--(4,1)--(4,3);
          
\foreach \x in {1,2} \draw[double,line width=.5pt] (\x,3)--(\x,0);
\foreach \y in {0,3} \draw[double,line width=.5pt] (1,\y)--(2,\y);

  \draw[->,line width=1.2pt] (0,5)--(0,3.1);
  \draw[->,line width=1.2pt] (1,5)--(1,3.1);
 \draw[->,line width=1.2pt]  (4,4)--(4,3.1);
  \draw[->,line width=1.2pt]  (0,0)--(.9,0);
       
 \draw[dashed]  (0,5)--(5,5)--(5,0) -- (2,0) (0,0)--(0,1);
  \draw[dashed]  (4,1)--(5,1) (5,2)--(0,2)(0,3)--(1,3); 
 \draw[dashed]  (3,5)-- (3,0) (2,3) --(5,3)(4,0)--(4,1);
  \draw[dashed]   (2,5)--(2,4)(4,4)--(5,4)(2,3)--(2,4); 
 
 \foreach \x in {0,1,2,3,4,5} 
    \foreach \y in {0,1,2,3,4,5} 
      { \node[B] () at (\x,\y) {};}
  
       \node[T,label=above:$s_2$](s2) at (0,5){}; 
        \node[T,label=left:$s_1$](s1) at (0,4){};  
         \node[T,label=above:$s_4$](s4) at (1,5){};                 
          \node[T](s3) at (1,3){};  \node() at (.7,2.7) {$s_3$};   
            \node() at (.7,3.3) {$w$}; 
   \node[T](t3) at (2,2){};    \node() at (2.3,1.7) {$t_3$};
 \node[T,label=below:$t_4$](t4) at (0,0){};           
 \node[T](t2) at (4,4){};   \node() at (4.3,4.3) {$t_2$};
                                          
     \node[T,label=above:$t_1$](t1) at (4,5){}; 
     
 \node[M]()at (4,3) {};  \node() at (4.3,2.7) {$t_2^\prime$}; 
 
\node[M,label=left:$s_2^\prime$]()at (0,3) {};   
\node() at (1.3,3.3) {$s_4^\prime$}; 
\node[M,label=below:$t_4^*$]()at (1,0) {};    
                       \node[txt]() at (3.6,4.4) {$P_1$};                                             
                         \node[txt]() at (1.5,.4) {$C$};
                          \node[txt]() at (3.4,1.4) {$X$};
                           \node[txt]() at (-.7,1) {$A(\ell)$};
\end{tikzpicture}                  

\end{center}
\caption{$\|NW\|=4$, $\|NE\|=2$}
\label{Q2}
\end{figure}
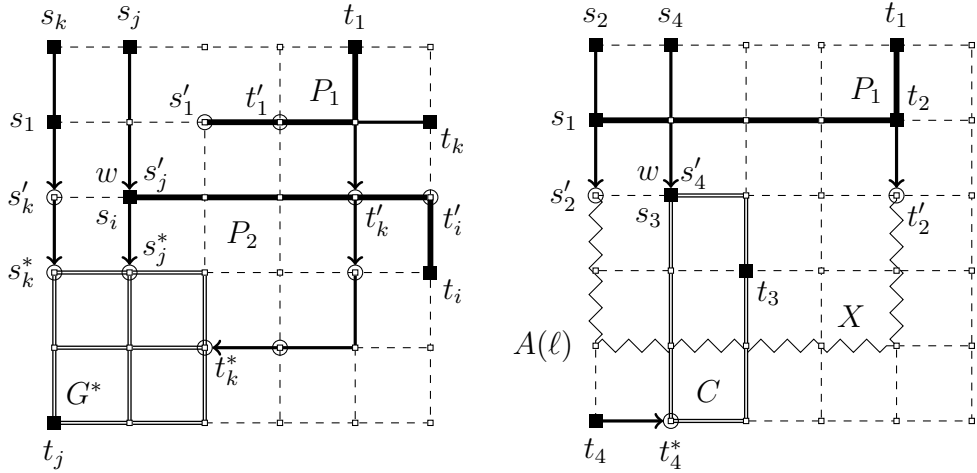
 Thus we assume that  
 $s_2^\prime,s_3^\prime,s_4^\prime$ are not distinct.  
 Let  $w\in A(3)\cap NW$ be the mate of two terminals of NW, that is $s_i^\prime=s_j^\prime=w$, for some $2\leq i<j\leq 4$ (actually, one of them is a terminal, $s_i^\prime=s_i$ or $s_j^\prime=s_j$). For $\|SW\|\leq 1$,
 $t_i$ or $t_j$ is a terminal in $NE\cup SE$, say $t_i\in NE\cup SE$. Now $t_i$ is mated to $t_i^\prime\in A(3)$  and a linkage for $\pi_i$ is completed by adding the $s_i^\prime,t_i^\prime$-path in
 $A(3)$. 
 Then  we mate the unlinked terminals of $NW\cup NE$  into the weakly $2$-linked quadrant SW to complete there the linkage for the remaining pairs $\pi_j,\pi_k$ (see in the left of Fig.\ref{Q2}).

 To tackle the last subcase we may assume that $t_1,t_k\in NE$ and
 $t_i,t_j\in SW$. W.l.o.g. let $i=3,j=4$,  that is 
 $s_3^\prime=s_4^\prime=w\in A(3)\cap NW$ and $t_3,t_4\in SW$. Let $P_1\subset NW\cup NE$ be a linkage for $\pi_1$, let
 $s_2^\prime\in (A(3)\setminus\{w\})\cap NW$ and $t_2^\prime \in A(3)\cap NE$, as obtained before. 
There is a row $A(\ell)$, $4\leq \ell\leq 6$ containing no terminal, thus a linkage for $\pi_2$ can be completed by
adding the path $Y$ from $s_2^\prime$ to $t_2^\prime$ in the union of their columns and $A(\ell)$. 

Define $C\subset SW$ to be the cycle bounded vertically by the two columns not containing $s_2^\prime$ and bounded horizontally by $A(3)$ and by row $A(5)$ if 
$\ell=6$ or by row $A(6)$ if $\ell\neq 6$. In this way we obtain a frame $[C,w]$. Since $t_3,t_4\notin X$, if $t_3$ and/or $t_4$ is not in $C$ it can be mated easily into $C$ along its row, thus the linkage for $\pi_3,\pi_4$ is obtained along the frame.  An example is shown in  the right of Fig.\ref{Q2}.\\

%%%%%%%%%%%%%%%%BBBB%%%%%%%%%%%%%%%%%%%%%%%%%%
Case B: there is a quadrant containing a pair. Assume that $NW$ 
contains a pair and it has the largest number of terminals with this property.\\

B.1.  $\|NW\|=7$ or $8$. The strategy consists in linking two pairs in NW and
mating the remaining terminals into $G-NW$ which is weakly $2$-linked by Lemma \ref{w2linked}. This plan works due to Lemma \ref{heavy78}.\\
 
B.2. $\|NW\|=6$. 

First we assume that NW consists of three pairs, let $\pi_1,\pi_2,\pi_3\in NW$. We extend NW into the grid $H\cong P_4\Box P_4$ by including the $7$-path $L\subset (A(4)\cup B(4))$ between $(1,4)$ and $(4,1)$.
By Lemma \ref{w2linked}, $H$ is  $3$-path-pairable,
therefore, there is a linkage in $H$ for $\pi_1,\pi_2$ and $\pi_3$. Removing the edges of $H$ from $G$
a connected graph remains, which contains a linkage for $\pi_4$.

Next we assume that $NW$ contains $\pi_1, \pi_2$ and the terminals $s_3,s_4$.
W.l.o.g. assume that $\|NE\|\geq \|SW\|$, and let  $A=NW\cap A(3)$,  $B=NW\cap B(3)$ and $x_0=(3,3)$. We apply Lemma \ref{heavy6} for $Q=NW$. We obtain a linkage for $\pi_1$ (or $\pi_2$ or both),
and a mating of the remaining terminals into distinct vertices of $A\cup B$  such that $B-x_0$
contains at most one mate.
% and $A-x_0$ has at most two mates. 
We extend these mating paths ending at $(A\cup B)\setminus\{x_0\}$ into at most three vertices of 
$\{(4,1),(4,2),(1,4),(2,4)\}$. Observe that after this step both quadrants NE and SW contain at most three (not necessarily distinct) terminals/mates.

Let $G^*=G-(A(1)\cup A(2)\cup B(1)\cup B(2))$, and let
 $L^*\subset (A(3)\cup B(3))$ be the $7$-path bounding $G^*$.
 By applying Lemma \ref{exit} (iii) twice, the 
 terminals/mates in NE and those in SW can be mated to distinct vertices of $L^*$  without using edges in $L^*$.
By Lemma \ref{3pp}, $G^*\cong P_4\Box P_4$ is $3$-path-pairable, hence the linkage for $\pi_2$ (or $\pi_1$) and $\pi_3,\pi_4$ can be completed in $G^*$.\\

B.3. $\|NW\|=5$.
First we assume that NW contains $\pi_1,\pi_2$ and a terminal $s_3$.
As in case B.2, we extend $NW$ into the grid $H\cong P_4\Box P_4$ by including 
the $7$-path $L\subset (A(4)\cup B(4))$ from $(1,4)$ to $(4,1)$. 
Let $s_3^*$ be any terminal-free vertex on $L$.
 Since  $H$ is $3$-path-pairable, there is a linkage in $H$ for the pairs $\pi_1,\pi_2$ and $\{s_3,s_3^*\}$.
Next we mate $s_3^*$ and the remaining terminals of $L$ into $G-H$ using edges from $G$ to $G-H$. 
By Lemma \ref{w2linked},  $G-H=A(5)\cup A(6)\cup B(5)\cup B(6)$ is weakly $2$-linked thus a linkage can be completed there for the pairs $\pi_3,\pi_4$.

Next we assume that $NW$ contains $\pi_1$ and the terminals  $s_2,s_3,s_4$. W.l.o.g. assume that $\|NE\|\geq\|SW\|$, and apply Lemma \ref{heavy5} with $NW$ to obtain a linkage for $\pi_1$
and mates $s_2^\prime,s_3^\prime\in A(3)$, $s_4^\prime\in B(3)\cap NW$.
For $\|NE\|\leq 2$ the solution is completed similarly to the one in B.2 as above.

For  $\|NE\|=3$ we extend the mating paths from
$s_2^\prime,s_3^\prime$ to vertices $s_2^*,s_3^*\in A(4)\cap SW$ and the mating path from 
$s_4^\prime$ to $s_4^{*}\in B(4)\cap NE$.
%\in \{(1,4),(2,4)\}$ (followed the shift $s_4^\prime\mapsto (2,3)$ for $s_4^\prime=(3,3)$). 
If $s_4^{*}\notin\{t_2,t_3,t_4\}$, then we apply Lemma  \ref{boundary} to 
obtain a linkage for $\pi_4$ and the mating of $t_2,t_3$ to $t_2^*,t_3^*\in A(4)$. 

Assume that $s_4^{*}=t_i$, for some $2\leq i\leq 4$. 
If $s_4^*=t_4$, then a linkage is obtained for $\pi_4$, and we mate $t_2,t_3$ to $t_2^*,t_3^*\in A(4)$.  W.l.o.g.  let  $s_4^*=t_2=w$. For $w=(3,4)$ we take a $w,t_4$-path in the weakly $2$-linked NE to complete the  linkage for $\pi_4$, and we mate $t_3$ to $t_3^\prime\in A(3)\cap NE$; then we mate $t_2,t_3^\prime$ to 
$t_2^*,t_3^*\in A(4)$. For $w=(2,4)$
we mate $t_2$ into $(4,4)$ along $B(4)$. Then we take a $w,t_4$-path in the weakly $2$-linked $NE\setminus (3,3)$ to complete the  linkage for $\pi_4$, and we mate $t_3$  to 
$t_3^*\in A(4)$. 
In each case we have $s_2^*,t_2^*,s_3^*,t_3^*\in A(4)$, thus  a linkage for $\pi_2,\pi_3$ can be completed in the 
weakly $2$-linked halfgrid $SW\cup SE$.\\

B.4:  $\|NW\|\leq 4$. Recall that NW contains a pair, say $\pi_1$, and $NW$ has the largest number of terminals with this property.

B.4.1. If $\|NW\|=2$ or $3$, then it contains one pair, say $\pi_1$, and by the choice of NW, we have $\|NE\|\leq 3$.
 Applying Lemma \ref{exit} (ii)  for NW and (iii) for NE, there is a linkage in NW for $\pi_1$ and there are distinct matings for the other terminals of $NW\cup NE$ into distinct vertices of
  $A(3)$ without using edges of $A(3)$. A  linkage
   for $\pi_2,\pi_3,\pi_3$ can be completed in the grid $G-(A(1)\cup A(2))\cong P_4\Box P_6$ which is $3$-path-pairable by Lemma \ref{3pp}.
 
B.4.2. Let $\|NW\|=4$. If NW contains two pairs then their linkage can be done in NW
and the linkage of the other two pairs in $G-NW$, since both $Q$ and $G-NW$ are $2$-path-pairable, by Lemma \ref{w2linked}. Thus we may assume that NW contains $\pi_1$ and terminals $s_2,s_3$. 
We distinguish cases where $\pi_4$ is contained by some quadrant $Q\neq NW$ or $s_4,t_4$ are in distinct quadrants.  

If $\pi_4\subset Q$, then we may assume, by symmetry, that $Q=$NE or SE. For $\pi_4\subset NE$,
we apply Lemma \ref{boundary} twice. The pair $\pi_1$ is linked in NW and $s_2,s_3$ are mated into $A(3)\cap NW$; the pair $\pi_4$ is linked in NE and the remaining terminals are mated to $A(3)\cap NE$.
Then the four (distinct) terminals/mates from $A(3)$ are mated further into  
 $SW\cup SE$ which is weakly $2$-linked by Lemma \ref{w2linked}. Thus 
 a linkage for $\pi_2, \pi_3$ can be completed in $SW\cup SE$.

Assume next that  $\pi_4\subset SE$. We apply Lemma \ref{boundary} with NW to obtain a linkage for $\pi_1$ and mates of $s_2,s_3$ through the appropriate boundary of NW into the neighboring quadrants. For $s_j$, $j=2,3$, we set $\psi(s_j)\subset A(3)$ if $t_j\in SW$, and 
$\psi(s_j)\subset B(3)$ if $t_j\in NE\cup SE$.
Using Lemma \ref{boundary} with SE
we take a linkage  for $\pi_4$ and mate $t_j\in SE$ into the neighboring quadrant where  $s_j$ is mated from NW.
Then we obtain the linkage for $\pi_2,\pi_3$ in the weakly $2$-linked quadrants SW and/or NE.

The remaining cases, where $\pi_4$ does not belong to any quadrant are listed in Fig.\ref{1P2S}. 

\begin{figure}[htp]\begin{center}
\tikzstyle{txt}  = [circle, minimum width=1pt, draw=white, inner sep=0pt]
\begin{tikzpicture}	
\draw (0,0) -- (1,0) -- (1,1) -- (0,1) -- (0,0);	
\draw (1.2,0) -- (2.2,0) -- (2.2,1) -- (1.2,1) -- (1.2,0);	
\draw (0,1.2) -- (1,1.2) -- (1,2.2) -- (0,2.2) -- (0,1.2);	
\draw (1.2,1.2) -- (2.2,1.2) -- (2.2,2.2) -- (1.2,2.2) -- (1.2,1.2);	
\node[txt](NW) at (.5, 1.9){$s_1$ $t_1$};
\node[txt]() at (.5, 1.45){$s_2$ $s_3$};
\node[txt](SW) at (.5, .7){$\emptyset$};
\node[txt](NE) at (1.7, 1.9){$s_4$};
\node[txt](SE) at (1.7, .7){$t_2$  {$t_3$}};
\node[txt]() at (1.7, .35){$t_4$};
\node[txt]() at (1.15, -.6){(I)};
\end{tikzpicture}
\hskip.8truecm		
\begin{tikzpicture}
\draw (0,0) -- (2.2,0) -- (2.2,1) -- (0,1) -- (0,0);	
\draw (0,1.2) -- (1,1.2) -- (1,2.2) -- (0,2.2) -- (0,1.2);	
\draw (1.2,1.2) -- (2.2,1.2) -- (2.2,2.2) -- (1.2,2.2) -- (1.2,1.2);	
\node[txt](NW) at (.5, 1.9){$s_1$ $t_1$};
\node[txt]() at (.5, 1.45){$s_2$ $s_3$};
\node[txt](NE) at (1.7, 1.9){$t_2$ $s_4$};
\node[txt](SE) at (1.1, .5){$t_3$  {$t_4$}};
\node[txt]() at (1.15, -.6){(II)};
\end{tikzpicture}
\hskip.8truecm
	\begin{tikzpicture}
\draw (0,0) -- (2.2,0) -- (2.2,1) -- (0,1) -- (0,0);	
\draw (0,1.2) -- (1,1.2) -- (1,2.2) -- (0,2.2) -- (0,1.2);	
\draw (1.2,1.2) -- (2.2,1.2) -- (2.2,2.2) -- (1.2,2.2) -- (1.2,1.2);	
\node[txt](NW) at (.5, 1.9){$s_1$ $t_1$};
\node[txt]() at (.5, 1.45){$s_2$ $s_3$};
\node[txt](NE) at (1.7, 1.9){$t_2$ $t_3$};
\node[txt]() at (1.7, 1.45){$s_4$};
\node[txt]() at (1.1, .5){$t_4$};
\node[txt]() at (1.15, -.6){(III)};
\end{tikzpicture}
			\hskip.8truecm
	\begin{tikzpicture}	
\draw (0,0) -- (1,0) -- (1,1) -- (0,1) -- (0,0);	
\draw (1.2,0) -- (2.2,0) -- (2.2,1) -- (1.2,1) -- (1.2,0);	
\draw (0,1.2) -- (1,1.2) -- (1,2.2) -- (0,2.2) -- (0,1.2);	
\draw (1.2,1.2) -- (2.2,1.2) -- (2.2,2.2) -- (1.2,2.2) -- (1.2,1.2);	
\node[txt](NW) at (.5, 1.9){$s_1$ $t_1$};
\node[txt]() at (.5, 1.45){$s_2$ $s_3$};
\node[txt](SW) at (.5, .7){$t_4$ };
\node[txt](NE) at (1.7, 1.9){ $s_4$};
\node[txt](SE) at (1.7, .7){ $t_2$ $t_3$};
\node[txt]() at (1.15, -.6){(IV)};
\end{tikzpicture}
\end{center}
\caption{$\|NW\|=4$, $\pi_1\subset NW$}
\label{1P2S}
\end{figure}
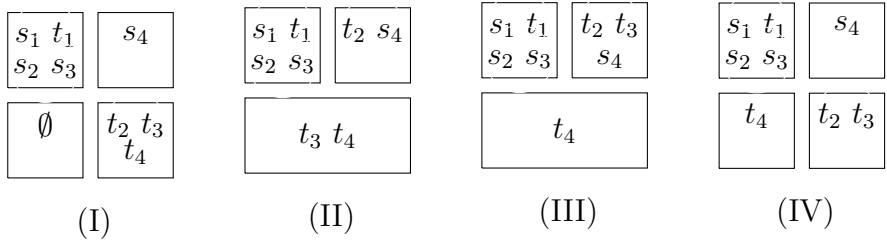

For type (I),  Lemma \ref{boundary} is used  for SE and for NW similarly as above. Thus we
obtain a linkage in NW for $\pi_1$, a linkage in SW for  $\pi_2,\pi_3$, and  a linkage in NE for $\pi_4$. 

For type (II),  we apply Lemma \ref{boundary} with NW to find a linkage for $\pi_1$ and to mate $s_2$ to $s_2^*\in B(4)\cap NE$ and to mate $s_3$ into $s_3^*\in A(4)\cap SW$. Then Lemma \ref{exit} (ii) is used with NE to complete a linkage in NW for $\pi_2$ and to mate $t_3$ to $t_3^*\in A(4)\cap SE$. 
The linkage for $\pi_3,\pi_4$ can be completed 
in  $SW\cup SE$ which is weakly $2$-linked, by Lemma \ref{w2linked}.

For types (III) and (IV), the linkage for  $\pi_2,\pi_3,\pi_4$ will be done by mating the terminals appropriately into 
 $G^*=G-(A(1)\cup A(2)\cup B(1)\cup B(2))\cong P_4\Box P_4$ which is  $3$-path-pairable, by Lemma \ref{3pp}.

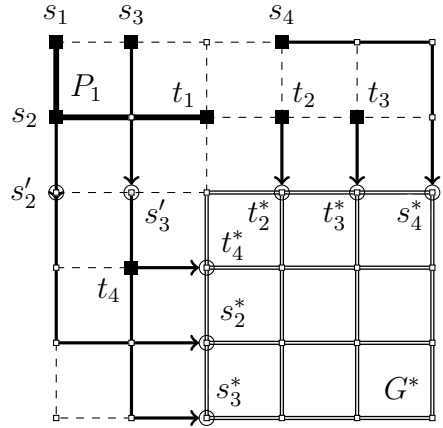
\begin{figure}[htp]\begin{center}
 \tikzstyle{M}  = [circle, minimum width=.5pt, draw=black, inner sep=2pt]
\tikzstyle{B} = [rectangle, draw=black!, minimum width=1pt, fill=white, inner sep=1pt]
\tikzstyle{T} = [rectangle, minimum width=.1pt, fill, inner sep=2.5pt]
 \tikzstyle{V} = [circle, minimum width=1pt, fill, inner sep=1pt]

 \tikzstyle{C} = [circle, minimum width=1pt, fill=white, inner sep=2pt]
 \tikzstyle{A} = [circle, draw=black!, minimum width=1pt, fill, inner sep=1.7pt]

\tikzstyle{txt}  = [circle, minimum width=1pt, draw=white, inner sep=0pt]
\tikzstyle{Wedge} = [draw,line width=2.2pt,-,black!100]

\begin{tikzpicture}
         \draw[->,line width=1.2pt]   (1,2) -- (1,0) -- (1.9,0); 
          \draw[->,line width=1.2pt] (1,5) -- (1,3.1); 
           \draw[->,line width=1.2pt]  (1,3) -- (1,2) -- (1.9,2); 
               \draw[->,line width=1.2pt] (0,4) -- (0,2.9);            
      \draw[->,line width=1.2pt]  (0,3) -- (0,2)  (0,2) -- (0,1) -- (1.9,1);                   
            \draw[->,line width=1.2pt]  (3,5) -- (5,5) -- (5, 3.1); 
                     \draw[->,line width=1.2pt]  (3,4) -- (3,3.1); 
                        \draw[->,line width=1.2pt]  (4,4) -- (4,3.1);                         
         \draw[double,line width=.5pt] (2,0) -- (5,0) -- (5,3) --(2,3)--(2,0); 
   \draw[double,line width=.5pt] (3,0) -- (3,3)  (4,0) --(4,3); 
   \draw[double,line width=.5pt] (2,1) -- (5,1) (2,2) --(5,2); 
           \draw[line width=2.2pt] (0,5) -- (0,4) --(2,4); 
 \draw[dashed] (0,5) -- (3,5) (2,4)-- (5,4) (0,3)--(2,3) (0,2)--(1,2) (0,1)--(0,0)--(1,0);
  \draw[dashed] (2,3) -- (2,5) (3,4)-- (3,5) (4,4)--(4,5);	
 \foreach \x in {0,1,2,3,4,5} 
    \foreach \y in {0,1,2,3,4,5} 
      { \node[B] () at (\x,\y) {};}
                                                 
          \node[T,label=above:$s_1$](s1) at (0,5){};
               \node[T,label=left:$s_2$](s2) at (0,4){};             
          \node[T,label=above:$s_3$](s3) at (1,5){};
          \node[T,label=above:$s_4$](s4) at (3,5){};
      \node[T](t1) at (2,4){};  \node[txt]() at (1.7,4.3){$t_1$} ; 
      \node[T](t3) at (4,4){}; \node[txt]() at (4.3,4.3){$t_3$} ; 
       \node[T](t2) at (3,4){};\node[txt]() at (3.3,4.3){$t_2$} ;             
      \node[T](t4) at (1,2){};\node[txt]() at (.7,1.7){$t_4$} ; 
                               
      \node[M,label=left:$s_2^\prime$](s2') at (0,3){} ;      
        \node[txt](t2*) at (2.7,2.7){$t_2^*$} ;  
    \node[M](s4*) at (5,3){} ;       
     \node[txt] () at (2.3,.35) {$s_3^{*}$};
          
            \node[M]() at (1,3){} ; \node[M]() at (2,1){} ; \node[M]() at (2,0){} ;
             \node[M]() at (2,2){} ; \node[M]() at (3,3){} ; \node[M]() at (4,3){} ;

           \node[txt]() at (2.35,1.4){$s_2^{*}$} ; 
             \node[txt]() at (1.35,2.7){$s_3^\prime$} ; 
                   \node[txt]() at (2.35,2.3){$t_4^{*}$} ; 
                    \node[txt]() at (3.7,2.7){$t_3^*$} ;        
   \node[txt]() at (4.7,2.7){$s_4^*$} ; 
   
          \node[txt]() at (4.6,.4){$G^*$} ; 
          \node[txt]() at (.4,4.4){$P_1$} ; 
\end{tikzpicture}

\end{center}
\caption{Solution for a pairing of type (III)}
\label{B32}
\end{figure}

 First we apply Lemma \ref{exit} (iii) for $NE$ to mate the terminals in $NE$ into distinct vertices of $A(3)$ 
 with mating paths not using edges in $A(3)$.
Next we use Lemma \ref{boundary} to obtain a linkage for $\pi_1$ and to mate
$s_j$ into $s_j^\prime\in A(3)\cap NW$, for $j=2,3$. If 
$s_j^\prime\neq (3,3)$, then we extend its mating path into $A(4)\cap SW$. Applying Lemma \ref{exit} (iii) for the terminals/mates in $SW$ we obtain the mates $s_2^{*}, s_3^{*}, t_4^*\in B(3)$. (In case of type (III) it is possible that $t_4\in SE$, when we just take $t_4^*=t_4$.) 
Then the linkage for  $\pi_2,\pi_3$, and $\pi_4$ can be completed, since all mating paths leading to $G^*$ are edge disjoint from $G^*$.  An example is shown in Fig.\ref{B32}.
 \end{proof}

\end{document}